\documentclass[amstex, 11pt]{amsart}
\usepackage[pdftex]{graphicx}
\usepackage{mathrsfs} 
\newtheorem{thm}{Theorem}[section]
\newtheorem{Pro}{Proposition}[section]
\newtheorem{lemma}{Lemma}[section]

\newtheorem*{CorN}{Corollary}
\newtheorem{Mythm}{Theorem}

\theoremstyle{definition}
\newtheorem{Def}{Definition}[section]

\theoremstyle{definition}
\newtheorem{Rem}{Remark}[section]
\numberwithin{equation}{section}
\begin{document}
\title[The spectrum of the number of geodesics]{On the spectrum of the number of geodesics \\ and tight geodesics in the curve complex}
\author{
    Ryo Matsuda \and  Kanako Oie  \and Hiroshige Shiga    
}
\address{Department of Mathematics, Faculty of Science, Kyoto University}
\email{matsuda.ryou.82c@st.kyoto-u.ac.jp}
\address{Graduate School of Humanities and Sciences, Nara Women's
University}
\email{yak\_oie@cc.nara-wu.ac.jp}
\address{Professor Emeritus at Tokyo Institute of Technology (Science Tokyo)} 
\email{shiga.hiroshige.i35@kyoto-u.jp}
\date{\today}    
\thanks{The first author was partially supported by JSPS Grant 23KJ1196.
The second author was partially supported by JST SPRING, Grant Number JPMJSP2115.
The third author was partially supported by the Ministry of Education, Science, Sports
and Culture, Japan;
Grant-in-Aid for Scientific Research (C), 22K03344.}

\begin{abstract}
Let $S$ be an oriented surface of type $(g, n)$.
We are interested in geodesics in the curve complex $\mathcal C(S)$ of $S$.
In general, two $0$-simplexes in $\mathcal C(S)$ have infinitely many geodesics connecting the two simplexes while another geodesics called tight geodesics are always finitely many.
On the other hand, we may find two $0$-simplexes in $\mathcal C(S)$ so that they have only finitely many geodesics between them.

In this paper, we consider the spectrum of the number of geodesics with length $d (\geq 2)$ in $\mathcal C(S)$ and tight geodesics, which is denoted by $\mathfrak{Sp}_d(S)$ and $\mathfrak{Sp}_d^T(S)$, respectively.

In our main theorem, it is shown that $\mathfrak{Sp}_d(S)\subset\mathfrak{Sp}_d^T(S)$ in general, but $\mathfrak{Sp}_2(S)=\mathfrak{Sp}_2^T(S)$.
Moreover, we show that $\mathfrak{Sp}_2(S)$ and $\mathfrak{Sp}_2^T(g, n)$ are completely determined in terms of $(g, n)$ and we obtain an isospectrum analogue for the curve complex of closed surfaces.
\end{abstract}
\maketitle
\section{Introduction}
Let $S$ be an oriented surface of type $(g, n)$, that is, it is of genus $g$ with $n$ punctures.
A simple closed curve on $S$ is called \emph{essential} if it is non-trivial and does not bound a puncture of $S$.

Other than exceptional surfaces, the set of homotopy classes of essential simple closed curves on $S$ makes a simplicial complex called \emph{the curve complex} of $S$ denoted by $\mathcal C(S)$ or $\mathcal C(g, n)$.
In the curve complex $\mathcal C(S)$, a $0$-simplex is a free homotopy class of an essential simple closed curve  on $S$.
Two $0$-simplexes of $\mathcal C(S)$ is joined by an edge if they admit disjoint representatives.

We shall abuse notation by writing $\alpha$ to mean both the simple closed curve $\alpha$ and its homotopy class.
By giving the length 1 to an edge in $\mathcal C(S)$, we see that $\mathcal C(S)$ is a metric space and we denote by $d_{\mathcal C(S)}$ the distance in $\mathcal C(S)$ given by the metric.
A path $\Gamma$ from $\alpha$ to $\beta$ in $\mathcal C(S)$ is called a \emph{geodesic} if the length of $\Gamma$ is equal to $d_{\mathcal C(S)}(\alpha, \beta)$.

The curve complex is introduced by W. Harvey \cite{Harvey}.
After the introduction, the importance of the curve complex is widely recognized.
Especially, H. Masur and Y. Misky \cite{MM} showed that the curve complex is Gromov hyperbolic.
In \cite{MM2}, they introduced the notion of tight geodesics and established a finiteness theorem for such geodesics.
Their research for the curve complex  plays an important role in solving the famous ending lamination conjecture on hyperbolic 3-manifolds in a paper J. Brock, D. Canary and Y. Minsky \cite{BCM}.
R. C. H. Webb \cite{Webb} employed the spectrum of tight geodesics to estimate the stable length of mapping class group elements and he found an algorithm that computes the stable length in finite time. These observations highlight the importance of tight geodesics and their spectrum in studying the mapping class group.

In this paper, we consider a structure of the set of geodesics in the curve complex.
Generally, geodesics between two points in $\mathcal C(S)$ are not unique.
In fact, there are infinitely many geodesics in many cases.
On the other hand, we may find two points in $\mathcal C(S)$ such that there are only finitely many geodesics between them in $\mathcal{C}(S)$.

For a given $d\geq 2$, we consider the spectrum of the number of geodesics with length $d$ in $\mathcal C(S)$, which is denoted by $\mathfrak{Sp}_d(g, n)$ or $\mathfrak{Sp}_d(S)$.
That is, $\mathfrak{Sp}_d(S)$ is the set of the number of geodesics of length $d$ connecting two points in $\mathcal C(S)$.
More precisely, $k\in \mathbb N$ contains in $\mathfrak{Sp}_d(S)$ if and only if there are two points $\alpha, \beta\in \mathcal C(S)$ such that $d_{\mathcal C(S)}(\alpha, \beta)=d$ and $|\mathcal G_S(\alpha, \beta)|=k$, where $\mathcal G_S(\alpha, \beta)$ is the set of geodesics connecting $\alpha$ and $\beta$ in $\mathcal C(S)$ and $|A|$ is the cardinal number of a set $A$.

We also consider the spectrum of the number of tight geodesics (see Definition \ref{Def:Tight} for the definition of tight geodesics) of length $d$ which is denoted by $\mathfrak{Sp}_d^T(g, n)$ or $\mathfrak{Sp}_d^T(S)$.
Namely, $k\in \mathbb N$ contains in $\mathfrak{Sp}_d^T(S)$ if and only if there are two points $\alpha, \beta\in \mathcal C(S)$ such that $d_{\mathcal C(S)}(\alpha, \beta)=d$ and $|\mathcal G_S^T(\alpha, \beta)|=k$, where $\mathcal G_S^T(\alpha, \beta)$ is the set of tight geodesics connecting $\alpha$ and $\beta$ in $\mathcal C(S)$. 

\medskip

The set $\mathcal G_S(\alpha, \beta)$, can be an infinite set.
However, for some pairs of $\alpha, \beta$, $\mathcal G_S(\alpha, \beta)$ is a finite set.
Hence, the spectrum $\mathfrak{Sp}_d(S)$ contains some positive integers.
In fact, we know the followings:
\begin{thm}[Ido-Jang-Kobayashi \cite{IJK}]
\label{thm:IJK}
		\item  For each $d\in \mathbb N$, $\mathfrak{Sp}_d(S)\ni 1$.
\end{thm}
\begin{thm}[Oie \cite{O}]
\label{thm:O}
	\begin{enumerate}
		\item If $(g, n)=(1, n)$, $n\geq 3$ or $(g, n)$ with $g\ge 2, n\geq 1$, then $\mathfrak{Sp}_2 (g, n)\ni 2$.
		\item If $(g, n)=(2, n)$, $n\ge 2$ or $g\geq 3$, then $\mathfrak{Sp}_2(g, n)\ni 3$.
	\end{enumerate}
\end{thm}

Viewing those results, we are interested in the \lq\lq size\rq\rq of $\mathfrak{Sp}_d(g, n)$.
First, we observe a relationship between $\mathfrak{Sp}_d$ and $\mathfrak{Sp}_d^T$.

\begin{Mythm}
\label{Thm}
	For any $d\geq 2$, $\mathfrak{Sp}_d(g, n)\subset\mathfrak{Sp}_d^T(g, n)$.
\end{Mythm}

When $d=2$, we have completely determined the spectrums for geodesics and tight geodesics. 
\begin{Mythm}
\label{ThmI}
 $\mathfrak{Sp}_2(g, n)=\mathfrak{Sp}_2^T(g, n)$.
\begin{enumerate}
	\item If $3g+n\geq 7$,
	\begin{equation}
		\mathfrak{Sp}_2(g, n)=\mathfrak{Sp}_2^T(g, n)=\left \{1, 2, \dots , \left \lfloor \frac{3g+n}{2}\right \rfloor\right \}.
	\end{equation}
	\item $\mathfrak{Sp}_2(0, 5)=\mathfrak{Sp}_2^T(0, 5)=\{1\}$, $\mathfrak{Sp}_2(0, 6)=\mathfrak{Sp}_2^T(0, 6)=\{1, 2\}$.
	\item $\mathfrak{Sp}_2(1, 2)=\mathfrak{Sp}_2^T(1, 2)=\{1\}$, $\mathfrak{Sp}_2(1, 3)=\mathfrak{Sp}_2^T(1, 3)=\{1, 2\}$.
	\item $\mathfrak{SP}_2(2, 0)=\mathfrak{Sp}_2^T(2, 0)=\{1, 2\}$.
\end{enumerate}

\end{Mythm}

It might be well-known that if two curve complexes are isometric as simplicial complexes, then the associated surfaces are homeomorphic to each other.
It is shown as an application of a result in \cite{Ivanov}, for example. 
Here, we  may show the following isospectrum analogue for the curve complex of closed surfaces, which is an immediate consequence of Theorem \ref{ThmI}.

\begin{CorN}
	Let $S, S'$ be closed oriented surfaces whose genera are more than one.
	Then, the followings are equivalent.
	\begin{enumerate}
		\item $S$ and $S'$ are homeomorphic to each other;
		\item $\mathfrak{Sp}_2(S)=\mathfrak{Sp}_2(S')$;
		\item $\mathfrak{Sp}_2^T(S)=\mathfrak{Sp}_2^T(S')$;
		\item $\mathcal C(S)$ and $\mathcal C(S')$ are isometric as simplicial complexes.
	\end{enumerate}
\end{CorN}

The proof of Theorem \ref{Thm} is given in \S 3 and the proof of the statement (1) of Theorem \ref{ThmI} is given in \S 4.
After explaining our operations for the curve complex in \S 5, we show the statements (2), (3) and (4) in \S 6.
We also confirm that $\mathfrak{Sp}_2(g, n)=\mathfrak{Sp}_2^T(g, n)$ at the end of \S 6. 

At \S 7, we will give an alternative proof of (1) of Theorem \ref{ThmI} which is done by the constructive method given in \S 4.

\medskip

	{\bf Acknowledgement:} The authors thank Prof. Hidetoshi Masai for his valuable comments and suggestions.

\section{Preliminaries}
\subsection{Simple surfaces and sporadic surfaces}
\begin{Def}
	An oriented surface $S$ of type $(g, n)$ is called \emph{simple} if it does not contain essential simple closed curves.
We say that $S$ is \emph{sporadic} if it does not contain disjoint essential simple closed curves.
A surface homeomorphic to a simple (resp. sporadic) surface is also said to be simple (resp. sporadic).
\end{Def}
It is easy to determine the types of simple surfaces and sporadic surfaces.
\begin{lemma}
\label{lemma:simple-sporadic}
	An oriented surface $S$ of type $(g, n)$ is simple if and only if it is homeomorphic to a surface of type $(0, n)$ for $0\leq n\leq 3$, and
$S$ is sporadic if and only if it is homeomorphic to a surface of type $(0, n)$ for $0\leq n\leq 4$, or of type $(1, n)$ for $n=0, 1$.
\end{lemma}

If a surface $S$ is sporadic, then the curve complex $\mathcal C(S)$ is trivial.
Hence, in this paper, we always assume that $S$ is not sporadic.

\subsection{Tight geodesics}

\begin{Def}
\label{Def:Tight}
Let $L(\alpha, \beta)=[\alpha, a_1, \dots , a_{d-1}, \beta]\subset \mathcal C(S)$ be a geodesic of length $d\geq 2$ connecting $\alpha$ and $\beta$.
Then, $L(\alpha, \beta)$ is called a \emph{tight geodesic} if for any $i\in \{1, \dots , d-1\}$, $b\in \mathcal C(S)$ with $a_i\cap b\not=\emptyset$ intersects $a_{i-1}\cup a_{i+1}$, where $a_0=\alpha$ and $a_d=\beta$.
\end{Def}
The existence of tight geodesics between any two points in $\mathcal C(S)$ is proved by Masur-Minsky \cite{MM2}.
Moreover, they show the following finiteness result (see also Bowdich \cite{Bowditch} for more stronger results).
\begin{Pro}[\cite{MM2} Corollary 6.14]
\label{Pro:TightFinite}
	For any $\alpha, \beta\in \mathcal C(S)$, there are only finitely many tight geodesics connecting them.
\end{Pro}

\subsection{Filling curves}
\begin{Def}
	A family $\{a_1, \dots , a_k\}$ of simple closed curves on $S$ is called  a \emph{filling set} if any connected component of $S\setminus \cup_{i=1}^{k}a_i$ is simply connected or homeomorphic to a punctured disk.
\end{Def}

As we stated in the introduction, the curve complex $\mathcal C(S)$ is a geodesic metric space. Namely, for any distinct points $\alpha, \beta \in \mathcal C(S)$, there exists a path $\ell$ in $\mathcal C(S)$ connecting the two points with $\|\ell \|=d_{\mathcal C(S)}(\alpha, \beta)$, where $\|\ell\|$ is the length of $\ell$.

We may characterize a pair of essential curves to be a filling set in terms of the distance of the curve complex.

\begin{Pro}
\label{Pro:filling}
	Let $\{\alpha, \beta\}$ be a pair of essential simple closed curves on $S$.
	Suppose that $\alpha$ is not homotopic to $\beta$.
	Then, $\{\alpha, \beta\}$ is a filling set if and only if $d_{\mathcal C(S)}(\alpha, \beta)\geq 3$.
\end{Pro}
\begin{proof}
	If $\{\alpha, \beta\}$ is not a filling set, then there exists an essential simple closed curve $\gamma$ in $S\setminus (\alpha\cup\beta)$ which is not homotopic to $\alpha$ nor $\beta$.
	Hence, $d_{\mathcal C(S)}(\alpha, \beta)\leq 2$.
	
	Conversely, if $\{ \alpha, \beta\}$ is a filling set, then there are no essential simple closed curves in $S\setminus (\alpha\cup\beta)$ which is not homotopic to $\alpha$ nor $\beta$.
	Hence, we verify that $d_{\mathcal C(S)}(\alpha, \beta)\geq 3$.
\end{proof}

We may guarantee the existence of filling pairs on non-simple surfaces.
\begin{Pro}
\label{Pro:ExistFilling}
	If $S$ is non-simple, then there exist $\alpha, \beta\in \mathcal C(S)$ such that $\{\alpha, \beta\}$ is a filling set.
\end{Pro}
\begin{proof}
	From Theorem \ref{thm:IJK}, there exist $\alpha, \beta\in \mathcal C(S)$ such that $d_{\mathcal C(S)}(\alpha, \beta)=3$.
	It follows from Proposition \ref{Pro:filling} that $\{\alpha, \beta\}$ is a filling set.
\end{proof}

\begin{Rem}
	Masur-Minsky (cf. \cite{MM} Proposition 4.6) shows that for any $\alpha\in \mathcal C(S)$, there exists a sequence $\{\varphi_n\}_{n=1}^{\infty}$ in the mapping class of $S$ such that $d_{\mathcal C(S)}(\alpha, \varphi_n(\alpha))\to \infty$ as $n\to \infty$.
	Proposition \ref{Pro:ExistFilling} is also proved by this result.
\end{Rem}

\section{Proof of Theorem \ref{Thm}}
Theorem \ref{Thm} is directly reduced to the following lemma.
\begin{lemma}
	Let $\alpha, \beta\in \mathcal C(S)$ with $d_{\mathcal C(S)}(\alpha, \beta)=d\geq 2$.
	Then, the followings are equivalent:
	\begin{enumerate}
		\item $|\mathcal G_S (\alpha, \beta)|<\infty$;
		\item every geodesic in $\mathcal G_S(\alpha, \beta)$ is a tight geodesic.
	\end{enumerate}
	\end{lemma}
\begin{proof}
		From Proposition \ref{Pro:TightFinite}, we see that (2) implies (1).
		
		To show the converse, we assume that there is a geodesic $L=[\alpha, a_1, \dots , a_{d-1}, \beta]$ which is not a tight geodesic.
		Then, there exists $b\in \mathcal C(S)$ and $i\in \{1, \dots , d-1\}$ such that $b\cap a_i=\emptyset$ but $b\cap (a_{i-1}\cup a_{i+1})\not=\emptyset$.
		Then, the image of $a_i$ by the $n$-times Dehn twists $(n\in \mathbb N)$ with respect to $b$, which is denoted by $a_{ij}^{(n)}$, makes a new geodesic connecting $\alpha$ and $\beta$.
		It contradicts $|\mathcal G_S (\alpha, \beta)|<\infty$.
		Thus, we verify that $L$ has to be a tight geodesic.
	\end{proof}
	
	Now, we prove Theorem \ref{Thm}. 
	
	Let $k\in \mathbb N$ be an element of $\mathfrak{Sp}_d(S)$.
	Then, there exist $\alpha, \beta$ in $\mathcal C(S)$ such that $|\mathcal G_S(\alpha, \beta)|=k$.
	It follows from Lemma 3.1 that $\mathcal G_S(\alpha, \beta)\subset \mathcal G_S^T(\alpha, \beta)$.
	On the other hand, we have $\mathcal G_S(\alpha, \beta)\supset \mathcal G_S^T(\alpha, \beta)$ from the definitions.
	Thus, we conclude that $\mathcal G_S(\alpha, \beta)= \mathcal G_S^T(\alpha, \beta)$ and $k\in \mathfrak{Sp}_d(S)$.
	This implies that $\mathfrak{Sp}_d(S)\subset \mathfrak{Sp}_d^T(S)$and the proof is completed.

\section{Proof of Theorem \ref{ThmI} (Part I)}
Let $S$ be an oriented surface of type $(g, n)$ with $3g+n\geq 7$.
As the first step, we show the following:
\begin{Pro}
\label{3g-4+n}
$\mathfrak{Sp}_2^T(S)\subset \{1, 2, \dots, 3g-4+n\}$.	
\end{Pro}
\begin{proof}
	Suppose that $k\in \mathfrak{Sp}_2^T(S)$.
	Then, there exists $\alpha, \beta\in \mathcal C(S)$ such that $|\mathcal G_S^T (\alpha, \beta)|=k$.
	It follows from the definition of the distance in $\mathcal C(S)$ that there exist essential simple closed curves $\alpha, \beta, a_1, \dots , a_k$ on $S$ such that
	$[\alpha, a_i, \beta]$ $(i=1, \dots , k)$ are tight geodesics.
	We observe 
	\begin{equation*}
		\alpha\cap\beta\not=\emptyset, \alpha\cap a_i=\emptyset, \beta\cap a_i=\emptyset \quad (i=1, 2, \dots, k).
	\end{equation*}
	Then, we see that $a_i\cap a_j=\emptyset$ up to isotopy if $i\not= j$.
	Indeed, if $a_i\cap a_j\not=\emptyset$, then the image of $a_j$ by the Dehn twist with respect to $a_i$, denoted by $b_j^{(i)}$, makes a new geodesic $[\alpha, b_j^{(i)}, \beta]$ and it is also a tight geodesic.
	This contradicts $|\mathcal G_S^T (\alpha, \beta)|=k$.
	
	Therefore, $A:=\{\alpha, a_1, \dots , a_k\}$ is a set of mutually disjoint essential simple closed curve.
	Thus, $|A|=k+1\leq 3g-3+n$ and we obtain the desired result.
\end{proof}

Now, let us consider $M_2^T(g, n):=\max\{k \mid k\in \mathfrak{Sp}_2^T(g, n)\}$.
It follows from Proposition \ref{3g-4+n} that $M_2^T(g, n)\leq 3g-4+n$.
First, we show $M_2^T(g, n)\geq \lfloor\frac{3g+n}{2}\rfloor\geq 3$.
To see this, we consider the following cases.
\begin{description}
	\item [(i)] both $g$ and $n$ are even numbers;
	\item [(ii)] both $g$ and $n$ are odd numbers.;
	\item [(iii)] $g$ is an odd number and $n$ is an even number;
	\item [(iv)] $g$ is an even number and $n$ is an odd number.
\end{description}
{\bf Case (i)}: In this case, $3g+n\geq 8$ since $3g+n$ is an even number.
We consider $\frac{g}{2}$ pairs of pants, $P_1, \dots , P_{\frac{g}{2}}$, and $\frac{n}{2}$ twice punctured disk, $D_1, \dots , D_{\frac{n}{2}}$ in $S$ so that they are mutually disjoint and $X:=S\setminus (\cup_{i=1}^{\frac{g}{2}}\overline{P_i}\cup\cup_{i=1}^{\frac{n}{2}}\overline{D_i})$ is a connected surface of genus $0$.
Then, it is easily seen that the number of the components of $\partial X$ is 
\begin{equation*}
	3\cdot\frac{g}{2}+\frac{n}{2}=\frac{3g+n}{2}=\left \lfloor \frac{3g+n}{2}\right \rfloor \geq 4,
\end{equation*}
and $X$ is not a simple surface.
Therefore, from Proposition \ref{Pro:ExistFilling} There exist $\alpha, \beta\in \mathcal C(S)$ such that $\{\alpha, \beta\}$ is a filling set in $X$.
We denote by $a_1, \dots , a_{\frac{3g+n}{2}}$ the set of components of $\partial X$.
Then, we see that $[\alpha, a_i, \beta]$ $(i=1, \dots, \frac{3g+n}{2})$ are geodesics in $\mathcal C(S)$ of length 2.
Moreover, they are tight geodesics since $P_i$ $(i=1, \dots , \frac{g}{3})$ and $D_i$ $(i=1, \dots , \frac{n}{2})$ are simple surfaces.
Thus, we verify that $M_2^T (g, n)\geq \left \lfloor \frac{3g+n}{2}\right \rfloor$.

\noindent
{\bf Case (ii)}: In this case, we also see that $3g+n\geq 8$.
We consider $\frac{g-1}{2}$ pairs of pants, $P_1, \dots , P_{\frac{g-1}{2}}$, $\frac{n-1}{2}$ twice punctured disk, $D_1, \dots , D_{\frac{n-1}{2}}$ and an annulus $A$ with one puncture in $S$ so that they are mutually disjoint and $X:=S\setminus (\cup_{i=1}^{\frac{g-1}{2}}\overline{P_i}\cup\cup_{i=1}^{\frac{n-1}{2}}\overline{D_i}\cup \overline A )$ is a connected surface of genus $0$.
Then, the number of the components of $\partial X$ is
\begin{equation*}
	3\cdot \frac{g-1}{2}+\frac{n-1}{2}+2=\frac{3g+n}{2}=\left \lfloor \frac{3g+n}{2}\right \rfloor\geq 4.
\end{equation*}
and $X$ is not a simple surface.
Therefore, from Proposition \ref{Pro:ExistFilling} there exist $\alpha, \beta\in \mathcal C(S)$ such that $\{\alpha, \beta\}$ is a filling set in $X$.
We denote by $a_1, \dots , a_{\frac{3g+n}{2}}$ the set of components of $\partial X$.
Then, we see that $[\alpha, a_i, \beta]$ $(i=1, \dots, \frac{3g+n}{2})$ are geodesics in $\mathcal C(S)$ of length 2.
Moreover, they are tight geodesics since $P_i$ $(i=1, \dots , \frac{g}{3})$, $D_i$ $(i=1, \dots , \frac{n}{2})$ and $A$ are simple surfaces.
Thus, we verify that $M_2^T (g, n)\geq \left \lfloor \frac{3g+n}{2}\right \rfloor$.

\noindent
{\bf Case (iii)}: We consider $\frac{g-1}{2}$ pairs of pants, $P_1, \dots , P_{\frac{g-1}{2}}$, $\frac{n}{2}$ twice punctured disk, $D_1, \dots , D_{\frac{n}{2}}$, and a non-separating simple closed curve $\gamma$ in $S$ so that they are mutually disjoint and $X:=S\setminus (\cup_{i=1}^{\frac{g-1}{2}}\overline{P_i}\cup\cup_{i=1}^{\frac{n}{2}}\overline{D_i}\cup\gamma )$ is a connected surface of genus $0$.
Then, we see that $X$ is not a simple surface and we may take $\alpha, \beta\in \mathcal C(S)$ so that $\{\alpha, \beta\}$ is a filling set in $X$ as above.

We also see that the boundary curves of $P_i$ $(i=1, \dots , \frac{g-1}{2})$ and $D_i$ $(i=1, \dots \frac{n}{2})$ together with $\gamma$ give tight geodesics between $\alpha$ and $\beta$ of length 2.
The number of the geodesics is
\begin{equation*}
	3\cdot \frac{g-1}{2}+\frac{n}{2}+1=\frac{3g+n-1}{2}=\left \lfloor \frac{3g+n}{2}\right \rfloor ,
\end{equation*}
and we obtain  $M_2^T (g, n)\geq \left \lfloor \frac{3g+n}{2}\right \rfloor$.

\noindent
{\bf Case (iv)}: 
We consider $\frac{g}{2}$ pairs of pants, $P_1, \dots , P_{\frac{g}{2}}$, $\frac{n-1}{2}$ and twice punctured disk, $D_1, \dots , D_{\frac{n-1}{2}}$ in $S$ so that they are mutually disjoint and $X:=S\setminus (\cup_{i=1}^{\frac{g}{2}}\overline{P_i}\cup\cup_{i=1}^{\frac{n-1}{2}}\overline{D_i})$ is a connected surface of genus $0$ with one puncture.
Then, the number of the boundary curves of $X$ is
\begin{equation*}
	3\cdot \frac{g}{2}+\frac{n-1}{2}=\frac{3g+n-1}{2}=\left \lfloor \frac{3g+n}{2}\right \rfloor.
\end{equation*}

By taking a filling set $\{\alpha, \beta\}$ in $X$, we obtain $M_2^T (g, n)\geq \left \lfloor \frac{3g+n}{2}\right \rfloor$.

\medskip
In all cases above, any component of $S\setminus \overline X$ is not simple.
Hence, there are no geodesics other than $[\alpha, a_i, \beta]$ $(i=1, \dots , \lfloor\frac{3g+n}{2}\rfloor )$.
This implies that $\mathcal G_S(\alpha, \beta)=\mathcal G_S^T(\alpha, \beta)$ and we have the following:
\begin{lemma}
\label{lemma:Both are same}
	There exist $\alpha, \beta\in \mathcal C(S)$ such that $|\mathcal G_S(\alpha, \beta)|=|\mathcal G_S^T(\alpha, \beta)|=\lfloor\frac{3g+n}{2}\rfloor$.
\end{lemma}

\medskip

Next, we show $M_2^T (g, n)\leq \left \lfloor \frac{3g+n}{2}\right \rfloor$.

We take $\alpha, \beta\in \mathcal C(S)$ so that $|\mathcal G_S^T (\alpha, \beta)|=M_2^T(g, n)$.
Then, there exist $a_1, \dots , a_{M_2^T}$ in $\mathcal C(S)$ such that $[\alpha, a_i, \beta]$ $(i=1, \dots, M_2^T)$ are tight geodesics.

	As we have seen in the proof of Proposition \ref{3g-4+n}, $a_i\cap a_j=\emptyset$ $(i\not= j)$.
	So, we may take an annular neighborhood $N_i$ of $a_i$ $(i=1, \dots , M_2^T(g, n))$ sufficiently small so that
	\begin{equation*}
		N_i\cap \alpha=\emptyset, N_i\cap\beta=\emptyset , N_i\cap N_j=\emptyset \quad (i\not=j).
	\end{equation*}
	Let $X$ be a connected component of $S\setminus \cup_{i=1}^{M_2^T(g, n)}\overline{N_i}$ containing $\alpha\cup\beta$ and $U_1, \dots , U_{\ell}$ be the set of connected components of $S\setminus \overline X$.
Each $U_i$ is a surface of type $(g_i, m_i, n_i)$, where $g_i$ is the genus, $m_i$ is the number of borders and $n_i$ is the number of the punctures of $U_i$.
Note that $m_i\geq 1$ for any $i\in\{1, \dots , \ell \}$.
Then, $U_i$ belongs to one of the following types.
\begin{description}
	\item [Type I] a surface of type $(0, 2, 0)$;
	\item [Type II] a surface of type $(g_i, 1, n_i)$;
	\item [Type III] a surface of type $(g_i, m_i, n_i)\not=(0, 2, 0)$ for $m_i\geq 2$.
\end{description}

A surface $U_i$ belonging to Type I is an annulus and the core curve is homotopic to some $a_k$ $(k\in \{1, \dots , M_2^T(g, n)\})$.
If $U_i$ is of Type II, then the boundary curve is homotopic to some $a_k$ $(k\in \{1, \dots , M_2^T(g, n)\})$.
For $U_i$ of Type III, each boundary curve is homotopic to some $a_k$ $(k\in \{1, \dots , M_2^T(g, n)\})$.

From the construction of $X$, we see that all homotopy classes of such core curves and boundary curves are distinct to each other.
Moreover, we may show that they agree with the set of homotopy classes of $\{a_1, \dots , a_{M_2^T(g, n)}\}$.
Indeed, if a homotopy class of $\{a_1, \dots , a_{M_2^T(g, n)}\}$, say $a_k$ is not contained in the homotopy classes of them, then there exists $U_i$ such that it contains a curve $a_k'$ which is homotopic to $a_k$.
The surface $U_i$ should be of Type II or Type III.
Then, we may find an essential curve $b$ in $U_i$ which crosses $a_i'$.
It contradicts that each boundary curve makes a tight geodesic between $\alpha$ and $\beta$.
Hence, we obtain the following:
\begin{lemma}
\label{BoundaryOfX}
	The number of boundary curves of $X$ is $M_2^T(g, n)+N_1$, where $N_1$ is the number of $U_i$'s of Type I.
	Hence, the number of the boundary curves of $X$ is not less than $3$.
\end{lemma}
\begin{proof}
	Since each core curve of $U_i$ of Type I makes two boundary curves of $X$, we obtain the number $M_2^T(g, n)+N_1$.
	Since we have already seen that $M_2^T(g, n)\geq \lfloor\frac{3g+n}{2}\rfloor \geq 3$, the number $M_2^T(g, n)+N_1$ is also $\geq 3$.
\end{proof}

We may control the topological type of $X$.
		\begin{lemma}
		\label{lemma:XisPlanar}
		The component $X$ is a planar surface containing at most one puncture of $S$.
	\end{lemma}
\begin{proof}
	Suppose that the genus $g'$ of $X$ is positive.
	Then, there exists a non-separating simple closed curve $\gamma$ in $X$.
	Since the number of the boundary curves of $X$ is at least $3$ (Lemma \ref{BoundaryOfX}), $X_{\gamma}:=X\setminus\gamma$ is not simple.
	Hence, there exist essential simple closed curves $\alpha', \beta'$ in $X_{\gamma}$ such that $\{\alpha', \beta'\}$  is a filling set of $X$ (Proposition \ref{Pro:ExistFilling}).
	Therefore, $1$-skeltons $[\alpha', \gamma, \beta']$, $[\alpha', a_i, \beta']$ $(i=1, \dots , M_2^T(g, n))$ are geodesics connecting $\alpha', \beta'$ in $\mathcal C(S)$.
	Moreover, they are tight geodesics since $\{\alpha' , \beta'\}$ is a filling set.
	Hence, we obtain $(M_2^T(g, n)+1)$ tight geodesics.
	It contradicts the definition of $M_2^T(g, n)$ and we verify that $X$ is a planar surface.
	
	If $X$ contains two punctures, then we may take an essential simple closed curve $\gamma$ which surrounds the two punctures.
	The number of the boundary components of $X$ is at least $3$.
	Hence, $X_{\gamma}:= X\setminus\overline{D_{\gamma}}$ is not simple, where $D_{\gamma}$ is a domain bounded by $\gamma$ and containing the two punctures.
	By using exactly the same argument as above, we obtain a contradiction.
	
	Thus, we conclude that $X$ is a planar surface containing at most one puncture of $S$.	
\end{proof}
	
	Now, we count the number $M_2^T(g, n)$ from the topological type of $X$, that is $(g, n)$.
	
	If we attach $U_i$ of Type I to $X$, then the resulting surface is of genus one.
	
If we attach $U_i$ of Type II to $X$, then the resulting surface is of genus $g_i$.
We also see that for $U_i$ of Type II
\begin{equation}
\label{Type II}
	g_i\geq 1 \textrm{ or } n_i\geq 2
\end{equation}
since $a_k$ is an essential curve.

If we attach $U_i$ of Type III to $X$, then the resulting surface is of genus $g_i+m_i-1$.
Since $(g_i, m_i, n_i)\not=(0, 2, 0)$ and $m_i\geq 2$, we have
\begin{equation}
\label{Type III}
	g_i+m_i+n_i\geq 3
\end{equation}

Therefore, we obtain
\begin{eqnarray*}
	g&=&N_1+\sum_{\textrm{Tupe II}}g_i+\sum_{\textrm{Type III}}(g_i+m_i-1) \\
	n&=&\sum_{\textrm{Type II+Type III}}n_i +\varepsilon,
\end{eqnarray*}
and 
\begin{equation*}
	M_2^T(g, n)=N_1+N_2+\sum_{\textrm{Type III}}m_i,
\end{equation*}
where $N_1$ is the number of $U_i$'s belonging to Type I, $N_2$ is the number of $U_i$'s belonging to Type II, and $\varepsilon$ is the number of punctures in $X$.
Hence, $\varepsilon=0$ or $1$.

Then, we have
\begin{eqnarray*}
	3g+n&=&3\left\{N_1+\sum_{\textrm{Tupe II}}g_i+\sum_{\textrm{Type III}}(g_i+m_i-1)\right\}\\
	&+&\sum_{\textrm{Type II+Type III}}n_i +\varepsilon \\
	&=& 3N_1+\sum_{\textrm{Type II}}(3g_i+n_i)+\sum_{\textrm{Type III}}\{3(g_i+m_i-1)+n_i\}+\varepsilon \\
	&=& 3N_1+\sum_{\textrm{Type II}}(3g_i+n_i)+\sum_{\textrm{Type III}}(g_i+m_i+n_i-3) \\
	&+&2\sum_{\textrm{Type III}}(g_i+m_i)+\varepsilon .
\end{eqnarray*}
From (\ref{Type II}) and (\ref{Type III}), we have
\begin{equation*}
	3g+n\geq 3N_1+2N_2+2\sum_{\textrm{Type III}}m_i+\varepsilon\geq 2M_2^T(g, n)+\varepsilon.
\end{equation*}
Thus, we obtain $M_2^T(g, n)\leq \lfloor\frac{3g+n}{2}\rfloor$ as desired and we conclude that $M_2^T(g, n)= \lfloor\frac{3g+n}{2}\rfloor$.

On the other hand, we have already seen that $M_2(g, n)\leq M_2^T(g, n)$ from Theorem \ref{Thm} and there exist $\alpha, \beta\in \mathcal C(s)$ such that $|\mathcal G_S(\alpha, \beta)|=\lfloor\frac{3g+n}{2}\rfloor$ and $d_{\mathcal C(S)}(\alpha, \beta)=2$.
Therefore, we have $M_2(g, n)=M_2^T(g, n)$ and
\begin{equation*}
	\mathfrak{Sp}_2(g, n)\subset\mathfrak{Sp}_2^T(g, n)\subset\left\{1, \dots ,\left \lfloor\frac{3g+n}{2}\right\rfloor \right\}.
\end{equation*}

We will show all of them are the same.

\begin{Rem}
\label{Remark}
	In the last section, we give an inductive construction of $\alpha, \beta\in \mathcal C(S)$ with $|\mathcal G_{S}(\alpha, \beta)|=\left\lfloor \frac{3g+n}{2}\right\rfloor$.
	Hence, if once we know $M_2(g, n)\leq \left\lfloor \frac{3g+n}{2}\right\rfloor$, we may directly confirm that $M_2(g, n)=M_2^T(g, n)=\left\lfloor \frac{3g+n}{2}\right\rfloor$ without knowing $M_2(g, n) \geq \left\lfloor \frac{3g+n}{2}\right\rfloor$.
\end{Rem}

\medskip

Let $m\in \mathbb N$ be in $\{1, 2, \dots , \left\lfloor \frac{3g+n}{2}\right\rfloor-1\}$ given in Lemma \ref{lemma:Both are same}.
We consider $\alpha, \beta\in \mathcal C(S)$ with $d_{\mathcal C(S)}(\alpha, \beta)=2$ and $|\mathcal G_S (\alpha, \beta)|=\left\lfloor \frac{3g+n}{2}\right\rfloor$.
Then, there exist $a_1, \dots , a_{\left\lfloor \frac{3g+n}{2}\right\rfloor)} \in \mathcal C(S)$ such that they give the geodesics between $\alpha$ and $\beta$.

Let $X$ be a connected component of $S\setminus \cup_{i=1}^{m}a_i$ containing $\alpha, \beta$.
From the construction in Lemma \ref{lemma:Both are same}, every connected component of $S\setminus\overline{X}$ is a simple surface and $X$ is not simple.
Hence, it follow from Proposition \ref{Pro:ExistFilling} that there exist $\alpha', \beta'\in \mathcal C(X)$ such that $\{\alpha', \beta'\}$ is a filling set in $X$.
Also, we see that $\alpha'\cap a_i=\beta'\cap a_i=\emptyset$ for $i=1, \dots , m$.
This implies that $d_{\mathcal C(S)}(\alpha', \beta')=2$ and $|\mathcal G_S (\alpha', \beta')|\geq m$.
However, any essential simple closed curve in $S$ other than $a_1, \dots , a_m$ intersects $\alpha'$ or $\beta'$ since $\{\alpha', \beta'\}$ is a filling set in $X$ and any connected component of $S\setminus\overline X$ is a simple surface.
Therefore, we verify that $|\mathcal G_S (\alpha', \beta')|=m$ and $\mathfrak {Sp}_2(g, n)\ni m$.
Thus, we conclude that $\mathfrak {Sp}_2(g, n)=\{1, 2, \dots , \left\lfloor \frac{3g+n}{2}\right\rfloor\}$ and
\begin{equation*}
	\mathfrak{Sp}_2(g, n)=\mathfrak{Sp}_2^T(g, n)=\left\{1, \dots ,\left \lfloor\frac{3g+n}{2}\right\rfloor \right\}.
\end{equation*}

\section{Basic Operations}

In this section, we exhibit some basic operations on $S$ which give a new surface $S'$ so that $\mathfrak{Sp}_{2}(S')$ is counted by $\mathfrak{Sp}_2(S)$.
We will use those operations in later sections. 

Let $\alpha, \beta\in \mathcal C(S)$ with $d_{\mathcal C(S)}(\alpha, \beta)=2$.
Then, there exist representatives of $\alpha$ and $\beta$ with minimal intersections.
We use the same symbols $\alpha\in \alpha$ and $\beta\in \beta$ for the representatives and take an intersection point $p_0$ of them.
Our operations are done in a neighborhood of $p_0$.

\begin{description}
	\item [Operation I] This operation is done by adding punctures in $S$. First, we take $\alpha'\in \alpha$ so that $S\setminus (\alpha'\cup\beta)$ has two bigons. Then, we add one or two punctures in the bigons, and we have a new surface $S'$ (Figure \ref{Fig.Operation I}). 
 
\begin{figure}[htbp]
	\centering
	\includegraphics[width=12cm]{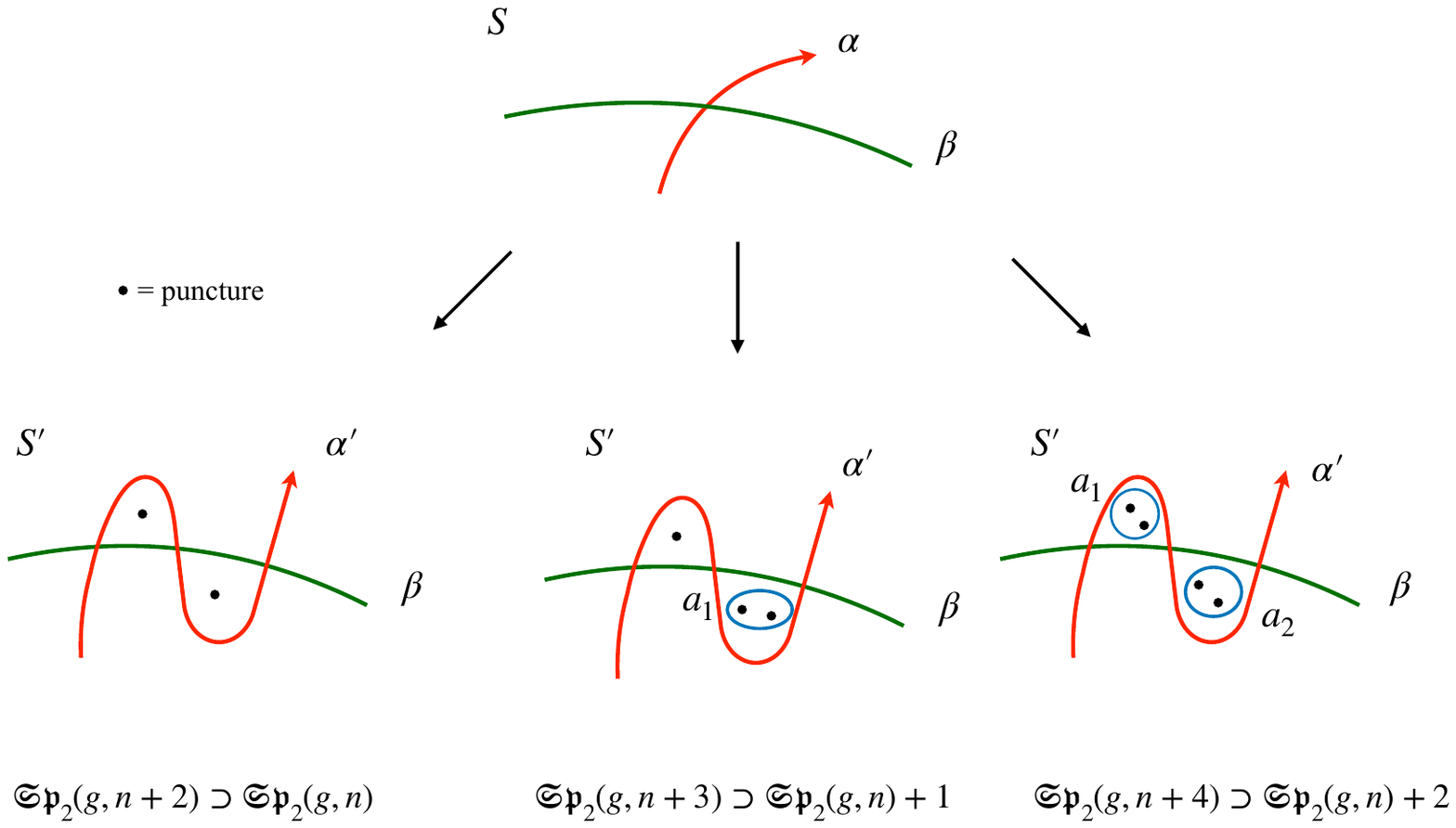}
	\caption{Operation I}
	\label{Fig.Operation I}
\end{figure}

Since we obtain a pair of pants or a punctured disk in $S'\setminus (\alpha'\cup\beta)$, simple closed curves $a_1, a_2$ in Figure \ref{Fig.Operation I} make new geodesics.
Hence, we have
\begin{equation}
\label{OpI-0}
	\mathfrak{Sp}_2(g, n+2)\supset \mathfrak{Sp}_2(g, n)
\end{equation}
\begin{equation}
\label{OpI-1}
	\mathfrak{Sp}_{2}(g, n+3)\supset \mathfrak{Sp}_2(g, n)+1,
\end{equation}
and
\begin{equation}
\label{OpI-2}
	\mathfrak{Sp}_{2}(g, n+4)\supset  \mathfrak{Sp}_2(g, n)+2,
\end{equation}
where $\mathfrak{Sp}_2(g, n)+k$ means
\begin{equation*}
	\mathfrak{Sp}_2(g, n)+k=\{n_1+k, \dots , n_{\ell}+k\}
\end{equation*}
for $\mathfrak{Sp}_2(g, n)=\{n_1, \dots , n_{\ell}\}$.


\medskip
\item [Operation II] In this operation, we add an annulus to $S$. First, we take $\alpha'\in \alpha$ so that $S\setminus (\alpha'\cup\beta)$ has two bigons as Operation I. Then, we add a cylinder connecting the bigons (Figure \ref{Fig.Operation II}). The resulting surface $S'$ is of type $(g+1, n)$ and the simple closed curve $a_1$ in Figure \ref{Fig.Operation II} makes a new geodesic.
Hence, we see
\begin{equation}
\label{OpII}
	\mathfrak{Sp}_2 (g+1, n)\supset \mathfrak{Sp}_2(g, n)+1.
\end{equation}

\begin{figure}[htbp]
	\centering
	\includegraphics[width=12cm]{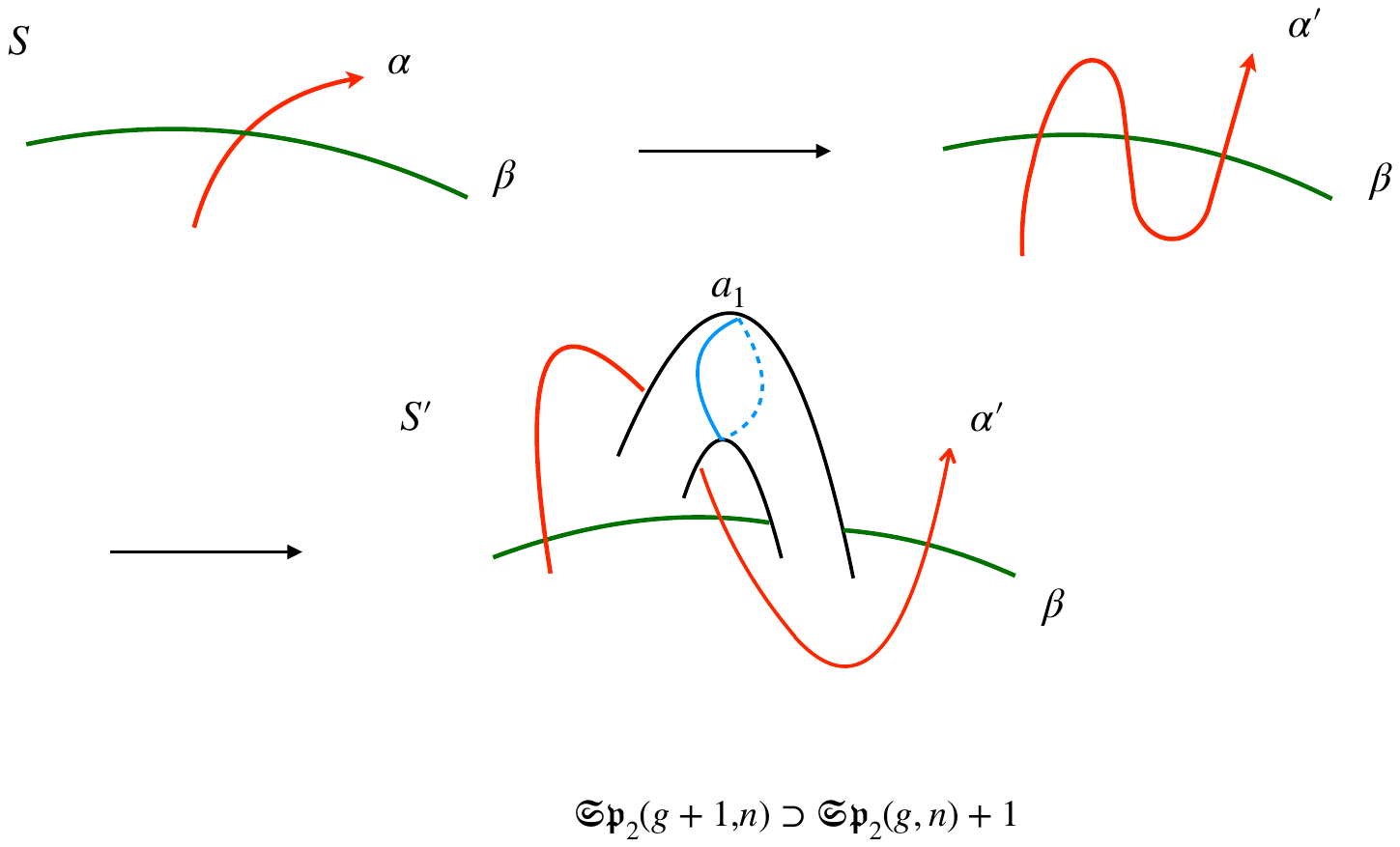}
	\caption{Operation II}
	\label{Fig.Operation II}
\end{figure}

\medskip

\item [Operation III] In this operation, we use a similar method to Operation II. We add a once punctured cylinder (= a pair of pants) instead of a cylinder in Operation II (Figure \ref{Fig.Operation III}). Hence, the resulting surface $S'$ is of type $(g+1, n+1)$. Then, we find a new pair of pants in $S'\setminus (\alpha'\cup\beta)$ bounded by subarcs of $\alpha'\cup\beta$ and the puncture. Therefore,
\begin{equation}
\label{OpIII}
	\mathfrak{Sp}_2(g+1, n+1)\supset \mathfrak{Sp}_2(g, n)+2.
\end{equation}
\begin{figure}[htbp]
\centering
	\includegraphics[width=12cm]{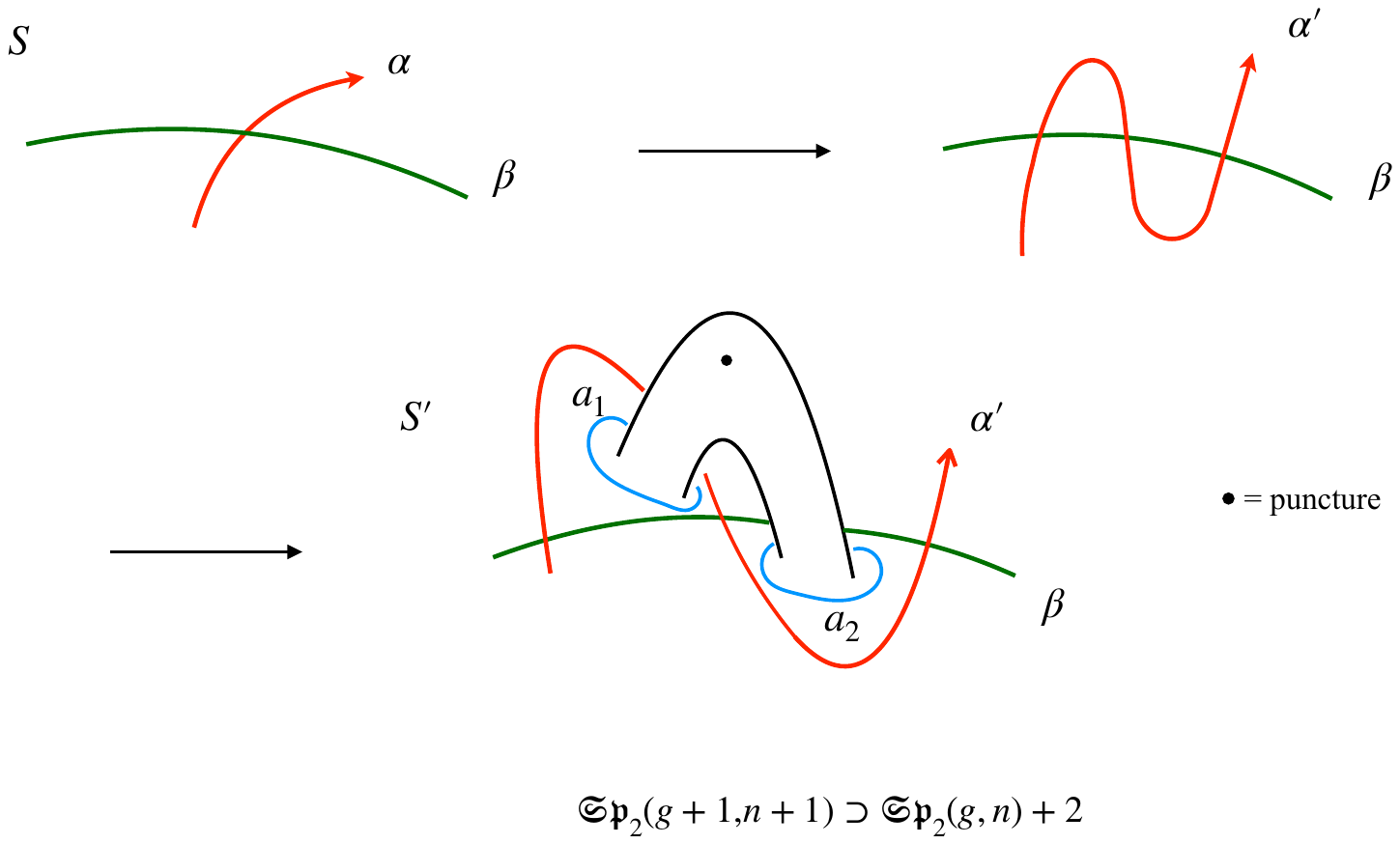}
	\caption{Operation III}
	\label{Fig.Operation III}
\end{figure}

\item [Operation IV] In this operation, we take $\alpha'\in \alpha$ so that $S\setminus (\alpha'\cup\beta)$ has six bigons. Choose three bigons in the six ones and put a pair of pants over them. Another pair of pants is also put on the other three bigons (Figure \ref{Fig. Operation IV}). The resulting surface $S'$ is of type $(g+4, n)$. Since we find new two pairs of pants in $S'\setminus (\alpha'\cup\beta)$ bounded by subarcs of $\alpha'\cup\beta$, we have
\begin{equation}
\label{OpIV}
	\mathfrak{Sp}_2 (g+4, n)\supset \mathfrak{Sp}_2(g, n)+6.
\end{equation}

\begin{figure}[htbp]
\centering
	\includegraphics[width=12cm]{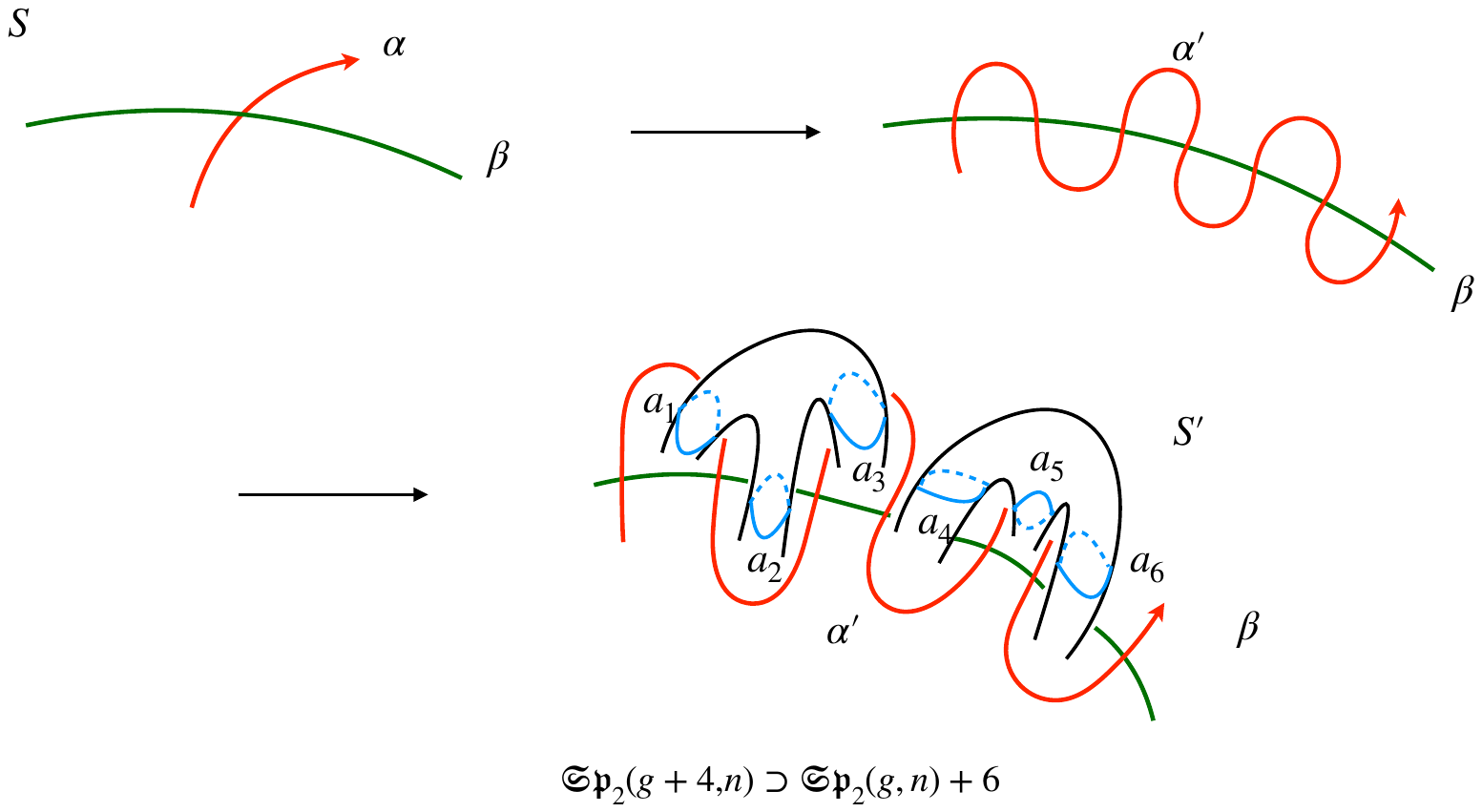}
	\caption{Operation IV}
\label{Fig. Operation IV}
\end{figure}
\end{description}

\subsection{An estimate of the spectrum}

Using the operations above, we show the following.

\begin{lemma}
\label{FirstEstimate}
For $g\geq 2$ and $n\geq 0$,
\begin{equation}
	\mathfrak{Sp}_2 (g, n)\supset \left \{1, 2, \dots , g\right \}.
\end{equation}	
\end{lemma}
\begin{proof}
	From Theorem \ref{thm:IJK}, we know that $\mathfrak{Sp}_2(g, n)\ni 1$.
	In fact, Figure \ref{(2,0)} shows that $\mathfrak{Sp}_2(2, 0)\supset\{1, 2\}$. 
	
	\begin{figure}[htbp]
	\centering
	\includegraphics[width=11cm]{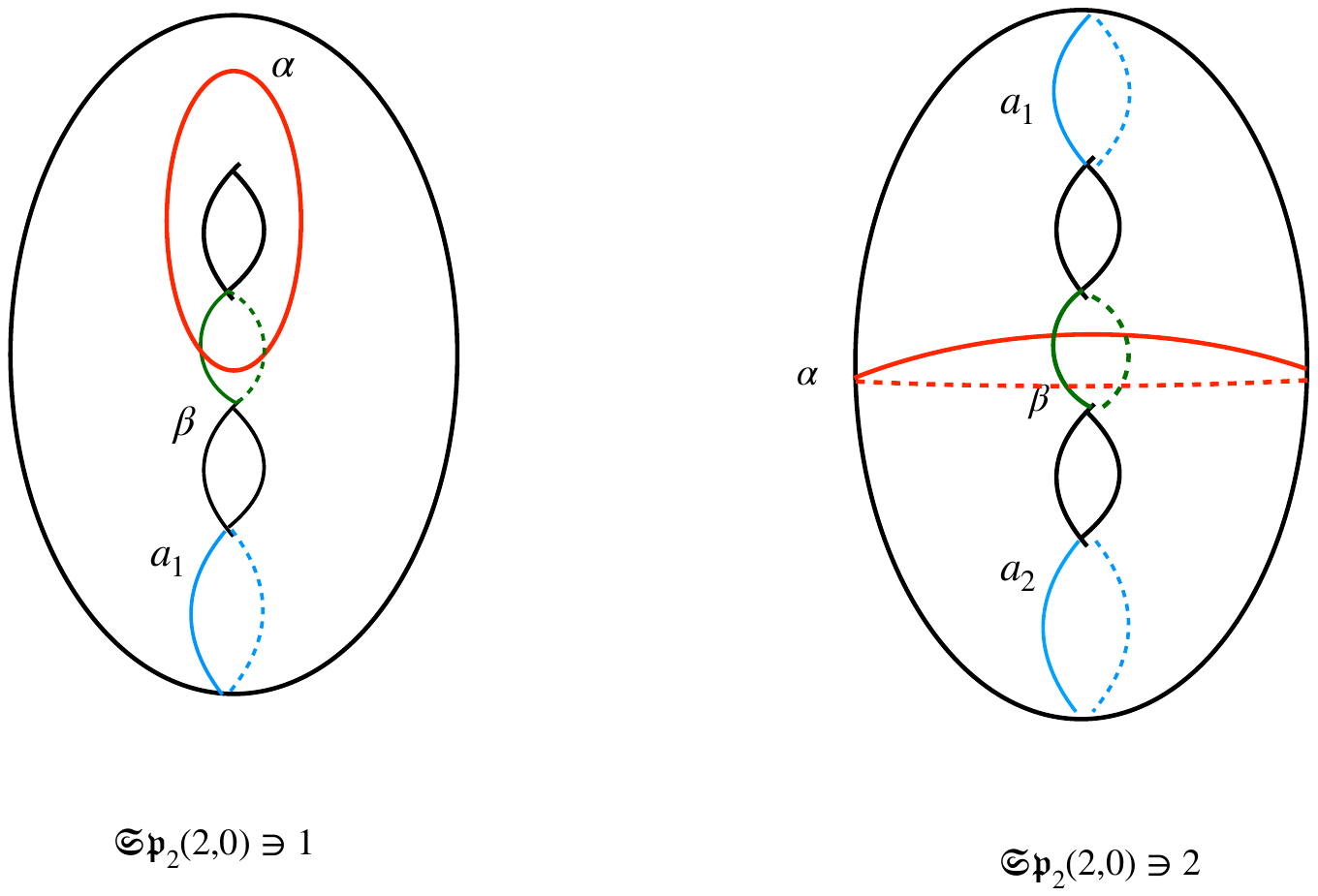}
	\caption{The (2, 0) spectrum}
	\label{(2,0)}
\end{figure}
	Hence, we have $\mathfrak{Sp}_2(3, 0)\supset \{1, 2, 3\}$ from (\ref{OpII}) of Operation II.
	Applying Operation II for surfaces of type $(3, 0)$ again, we obtain $\mathfrak{Sp}_2(4, 0)\supset\{1, 2, 3, 4\}$.
	Repeating this argument, we obtain $\mathfrak{Sp}_2 (g, 0)\supset \{1, \dots , g\}$.
	The statement is true for $n=0$ and $g\geq 2$.


	We consider the case where $n\geq 1$.
	We also use the same argument as above.
	Figure \ref{Fig.(2,1)} shows that $\mathfrak{Sp}_2(2, 1)\supset\{1, 2, 3\}$.
	
	\begin{figure}[htbp]
		\centering
		\includegraphics[width=12cm]{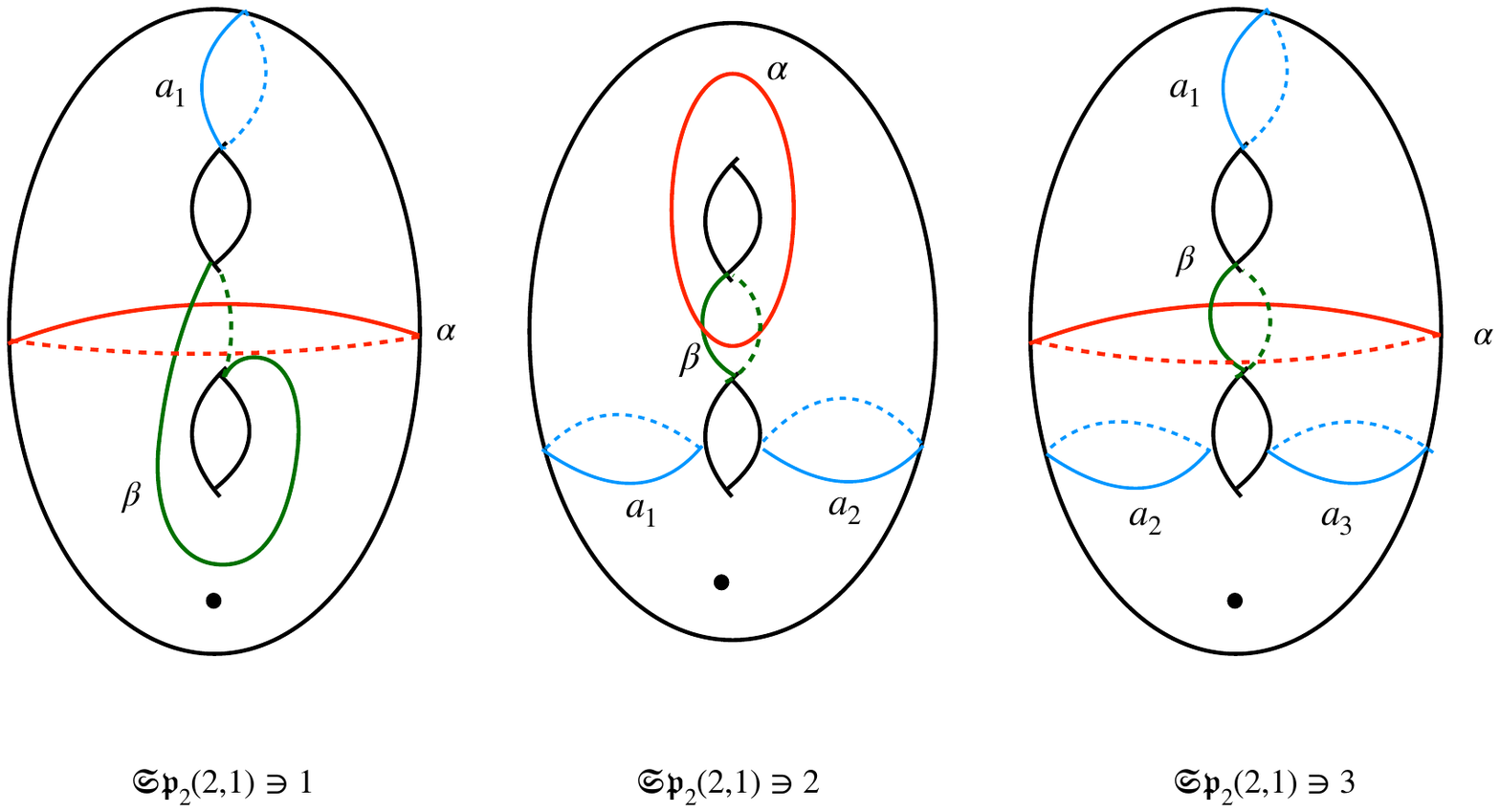}
		\caption{The $(2, 1)$ spectrum}
		\label{Fig.(2,1)}
	\end{figure}

	Hence, we have $\mathfrak{Sp}_2(3, 1)\supset\{1, 2, 3, 4\}\supset\{1, 2, 3\}$ from Operation II.
	Using Operation II repeatedly as above, we have $\mathfrak{Sp}_2(g, 1)\supset\{1, 2, \dots , g+1 \}\supset\{1, \dots , g\}$.
	Thus, the statement is true for $n=1$.
	
	
	Now, we use other operations to show surfaces of $n\geq 2$.
	We consider the following two cases to the  lemma.
	\begin{enumerate}
		\item $n=2m$ $(m\in \mathbb N)$: 
We begin with a surface $S'$ of type $(g, 0)$. From the above argument, we see that $\mathfrak{Sp}_2(S')\supset\{1, \dots , g\}$.
We construct $S$ of type $(g, n)$ by adding $n=2m$ punctures to $S'$.
Viewing (\ref{OpI-0}) of Operation I and applying Operation I by $m$ times, we verify that
\begin{equation*}
	\mathfrak{Sp}_2(S)\supset\mathfrak{Sp}_2(S').
\end{equation*}
Therefore, we see
\begin{equation*}
	\mathfrak{Sp}_2(g, n)\supset\{1, 2, \dots , g \}.
\end{equation*}
\item $n=2m+1$ $(m\in \mathbb N)$: 
We start with a surface $\dot S$ of type $(g, 1)$. We know that $\mathfrak{Sp}_2(\dot S)\supset\{1, \dots , g\}$.
We construct $S$ of type $(g, n)$ by adding $n-1=2m$ punctures to $\dot{S}$ and apply Operation I by $m$ times as (1). Just by the same reason as (1), we have
\begin{equation*}
	\mathfrak{Sp}_2(g, n)\supset\{1, 2, \dots , g \}.
\end{equation*}
	\end{enumerate}
	We have completed the proof.	
\end{proof}
\section{Proof of Theorem \ref{ThmI} (Part II)}
In this section, we show the statements of (2), (3) and (4) of Theorem \ref{ThmI}.
\begin{description}
	\item [Proof of (2)] Figure \ref{Fig.(0,5)} shows that $\mathfrak{Sp}_2(0, 5)\ni 1$.

\begin{figure}[htbp]
	\centering
	\includegraphics[width=11cm]{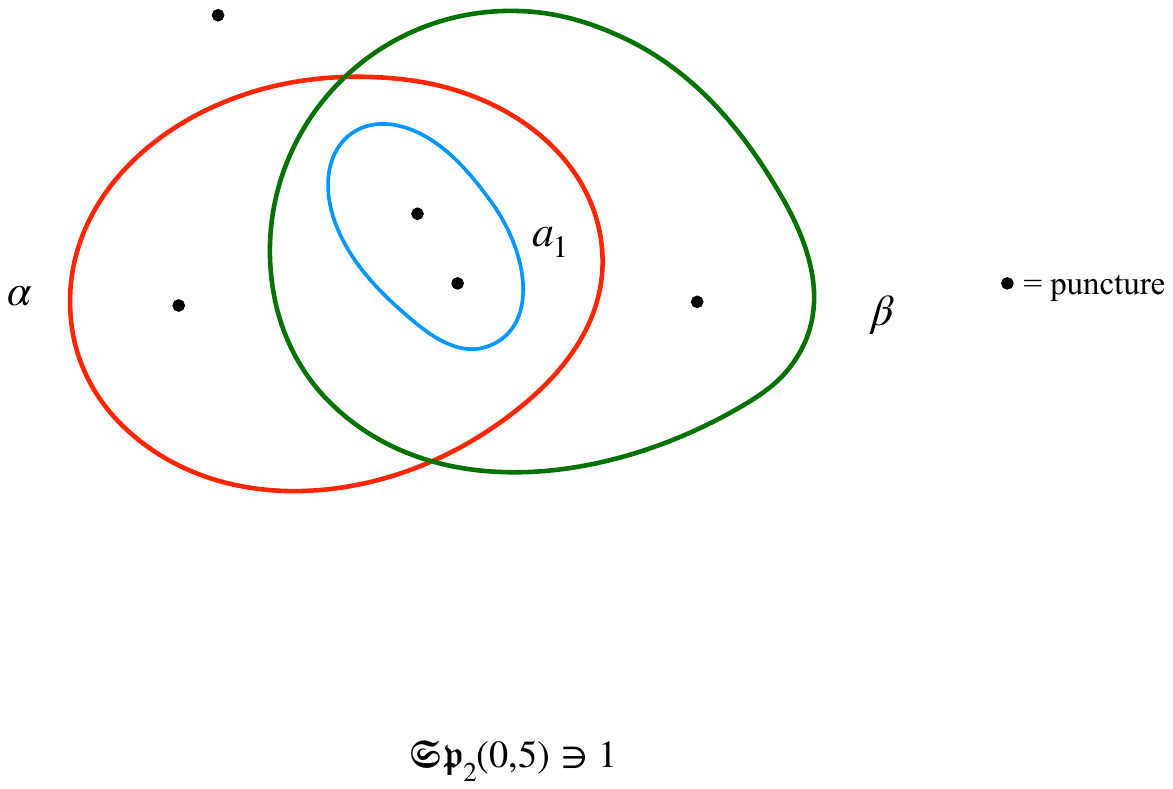}
	\caption{$(0, 5)$ spectrum}
	\label{Fig.(0,5)}
\end{figure}

\medskip

It follows from Proposition \ref{3g-4+n} that $|G_S(\alpha, \beta)|\leq M_2(g, n):=3g-4+n$ for a surface $S$ of type $(g, n)$.
Since $m(0, 5)=1$, we see that $\mathfrak{Sp}_2(0, 5)=\{1\}$.

We also see that $\mathfrak{Sp}_2(0, 6)\supset\{1, 2\}$ (Figure \ref{Fig.(0,6)}), and $m(0, 6)=2$.
\begin{figure}[htbp]
	\centering
	\includegraphics[width=11cm]{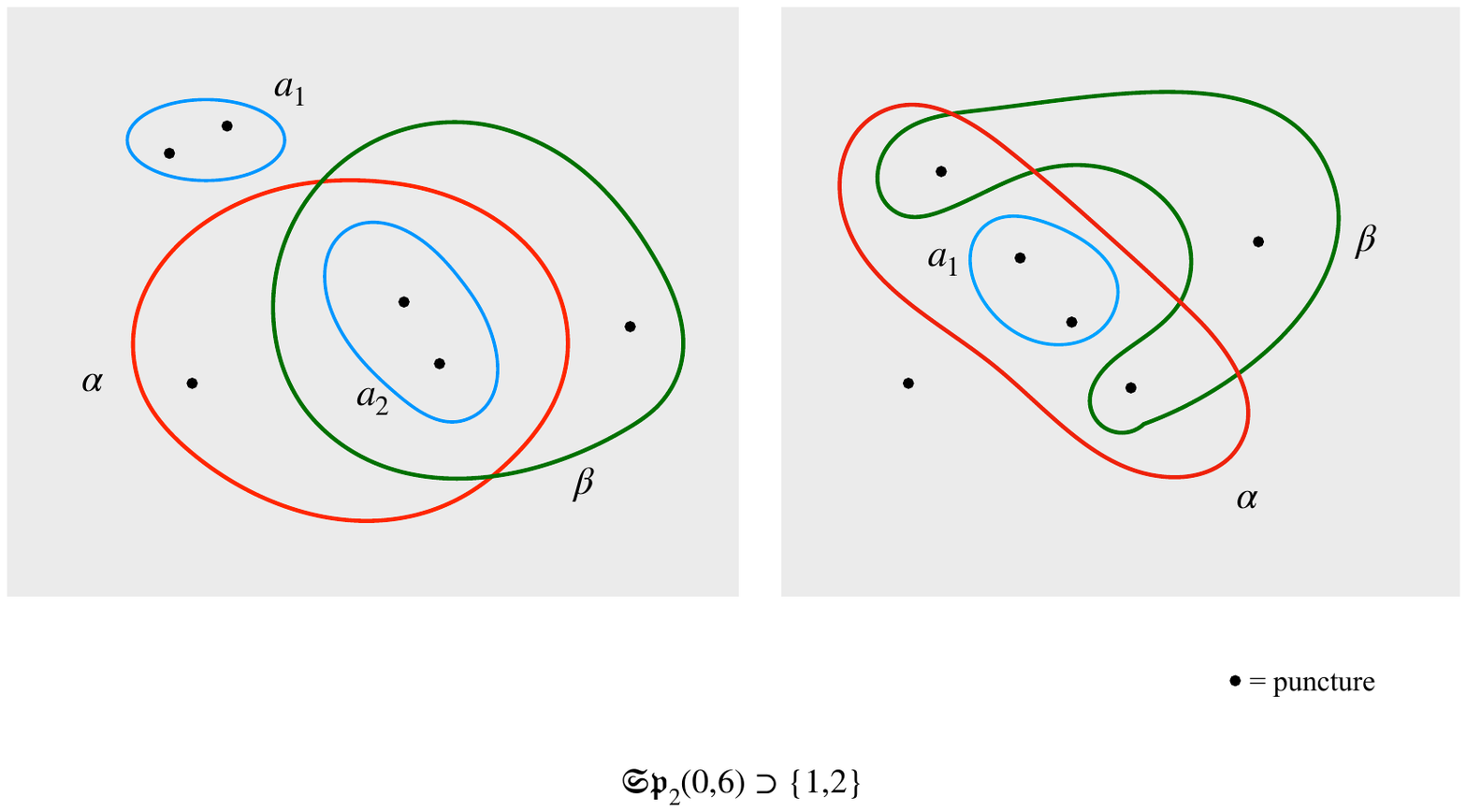}
	\caption{$(0,6)$ spectrum}
	\label{Fig.(0,6)}
\end{figure}
Hence, we conclude that $\mathfrak{Sp}_2(0, 6)=\{1, 2\}$.


\item [Proof of (3)] We see that $\mathfrak{Sp}_2(1, 2)\ni 1$ (Figure \ref{Fig.(1,2)}).
Since $m(1, 2)=1$, we see that $\mathfrak{Sp}_2(1, 2)=\{1\}$.

From Figure \ref{Fig.(1,3)}, we see that $\mathfrak{Sp}_2(1, 3)\supset\{1, 2\}$ and $m(1, 3)=2$.
Thus, we have $\mathfrak{Sp}_2(1, 3)=\{1, 2\}$.


\begin{figure}[htbp]
	\centering
	\includegraphics[width=12cm]{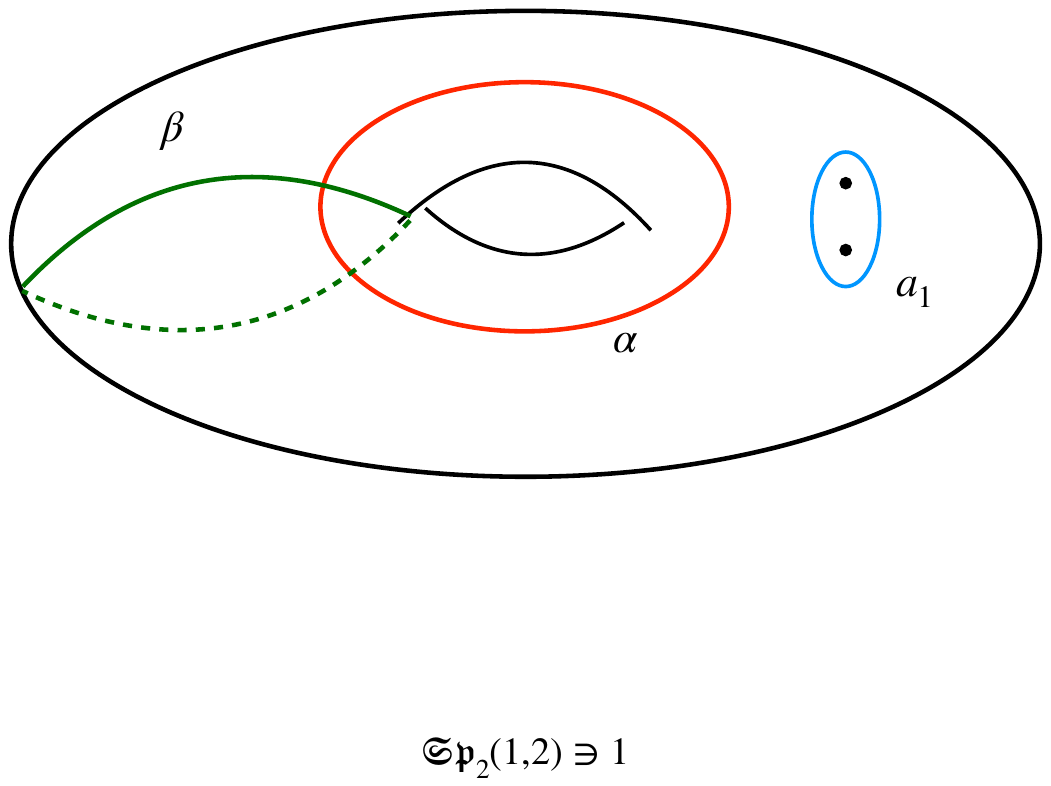}
	\caption{A surface of $(1, 2)$ type}
	\label{Fig.(1,2)}
\end{figure}
\begin{figure}[htbp]
	\centering
	\includegraphics[width=12cm]{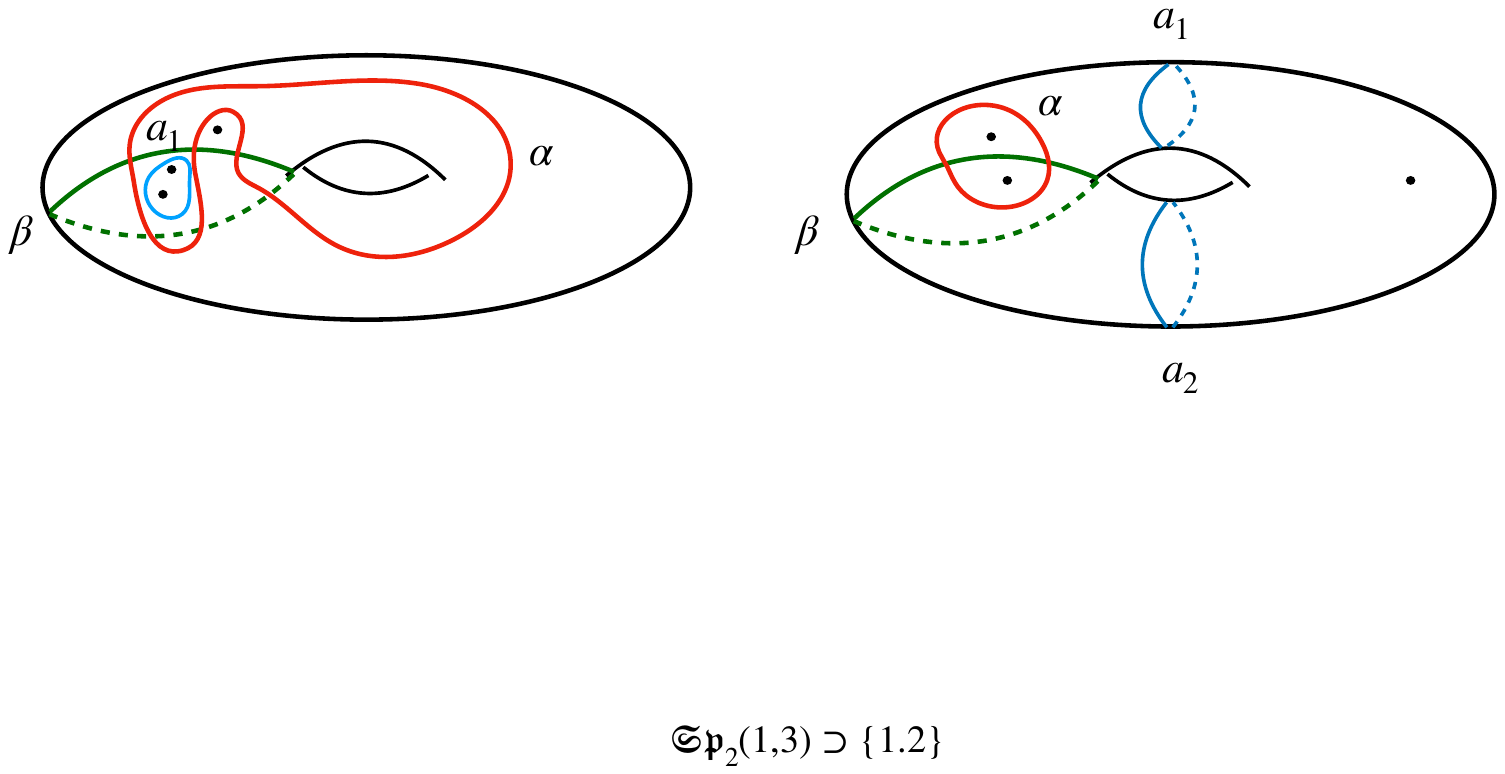}
	\caption{The $(1, 3)$ spectrum}
	\label{Fig.(1,3)}
\end{figure}

\item [Proof of (4)] We have already seen in Figure \ref{(2,0)} that $\mathfrak{Sp}_2(2, 0)\supset\{1, 2\}$, and $m(2, 0)=2$.
Hence, we conclude that $\mathfrak{Sp}_2(2, 0)=\{1, 2\}$.
\end{description}

It is easy to see that all geodesics given in this section are tight geodesics.
Hence, we verify that $\mathfrak{Sp}_2(g, n)=\mathfrak{Sp}_2^T(g, n)$ always holds.

\section{Constructive Approach (an alternative proof of Theorem \ref{ThmI})}

In this section, we will show that the statement (1) of Theorem \ref{ThmI} by using constructive methods.
More precisely, for a given $k\in \{1, 2, \dots , M_2(g, n)\}$, we exhibit a construction of $\alpha, \beta\in \mathcal C(S)$ with $|\mathcal G_S (\alpha, \beta)|=k$.
In the construction, we use the basic operations of \S 4.
Here, the statements are repeated while the method of the proofs are different from those in \S4.

\medskip

We begin with the planar case.
\begin{lemma}
\label{planar Lemma}
	If $n\geq 7$, then
	\begin{equation}
	\label{Lemma2}
		\mathfrak{Sp}_2(0, n)=\left \{1, 2, \dots , \left \lfloor \frac{n}{2}\right \rfloor\right \}.
	\end{equation}
\end{lemma}
\begin{proof}
	We show the lemma by using an inductive argument.
	Suppose that (\ref{Lemma2}) is valid for some $n\geq 7$.
	Then, by (\ref{OpI-0}) we have
	\begin{equation*}
		\mathfrak{Sp}_2(0, n+4)\supset\mathfrak{Sp}_2(0, n)=\left \{1, 2, \dots , \left \lfloor \frac{n}{2}\right \rfloor\right \}.
	\end{equation*}
	By (\ref{OpI-2}), we also have
	\begin{eqnarray*}
		\mathfrak{Sp}_2(0, n+4)&\supset &\mathfrak{Sp}_2(0, n)+2 \\
		&\supset &\left \{3, \dots , \left \lfloor \frac{n}{2}+2\right \rfloor\right \}=\left \{3,  \dots , \left \lfloor \frac{n+4}{2}\right \rfloor\right \}.
	\end{eqnarray*}
	Since $\lfloor \frac{n}{2}\rfloor\geq 3$, we verify that
	\begin{eqnarray*}
		\mathfrak{Sp}_2(0, n+4)&\supset&\mathfrak{Sp}_2(0, n)\cup(\mathfrak{Sp}_2(0, n)+2) \\
		&\supset &\left \{1, 2, \dots , \left \lfloor \frac{n+4}{2}\right \rfloor\right \}.
	\end{eqnarray*}
	We have already known that $\lfloor\frac{n+4}{2}\rfloor=M_2(0, n+4)$.
	Therefore, (\ref{Lemma2}) is valid for $n+4$ if it is valid for $n\geq 7$.
	
	From this argument, we just show that (\ref{Lemma2}) holds for $n=7, 8, 9, 10$ to show it for any $n\geq 7$.
	We show such cases by constructing concrete examples.
	\begin{itemize}
		\item $n=7$: Since $\mathfrak{Sp}_2(0, 5)=\{1\}$, it follows from (\ref{OpI-0}) that $\mathfrak{Sp}_2(0, 7)\ni 1$.
		The pictures of the left side in Figure \ref{Fig.7,8} show that $\mathfrak{Sp}_2(0, 7)\ni 2, 3$.
\begin{figure}[htbp]
	\centering
	\includegraphics[width=12cm]{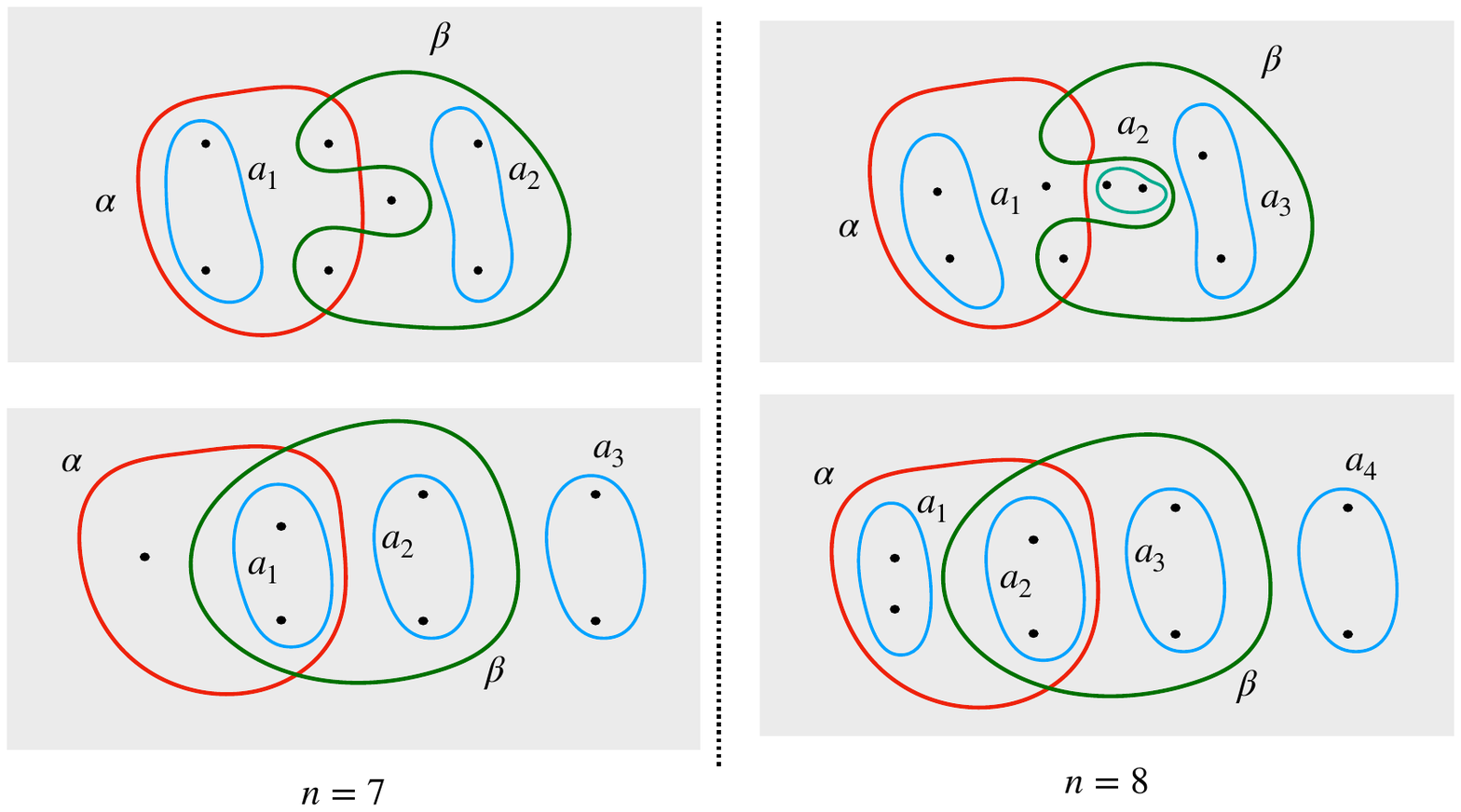}
	\caption{}
	\label{Fig.7,8}
\end{figure}
		We verify that
		\begin{equation*}
			\mathfrak{Sp}_2(0, 7)=\{1, 2, 3\}= \left \{1, 2,  \left \lfloor \frac{7}{2}\right \rfloor\right \}.
		\end{equation*}
		\item $n=8$: Since $\mathfrak{Sp}_2(0, 6)=\{1, 2\}$, it follows from (\ref{OpI-0}) that $\mathfrak{Sp}_2(0, 8)\supset \{1, 2\}$.
The picture of the right side in Figure \ref{Fig.7,8} show that $\mathfrak{Sp}_2(0, 8)\ni 3, 4$.
We verify that
		\begin{equation*}
			\mathfrak{Sp}_2(0, 8)=\{1, 2, 3, 4\}= \left \{1, 2, 3, \left \lfloor \frac{8}{2}\right \rfloor\right \}.
		\end{equation*}
		\item $n=9$: Since $\mathfrak{Sp}_2(0, 7)=\{1, 2, 3\}$, it follows from (\ref{OpI-0}) that $\mathfrak{Sp}_2(0, 9)\supset \{1, 2, 3\}$.
The picture of the left side in Figure \ref{Fig.9,10} shows that $\mathfrak{Sp}_2(0, 9)\ni 4$.
\begin{figure}[htbp]
	\centering
	\includegraphics[width=12cm]{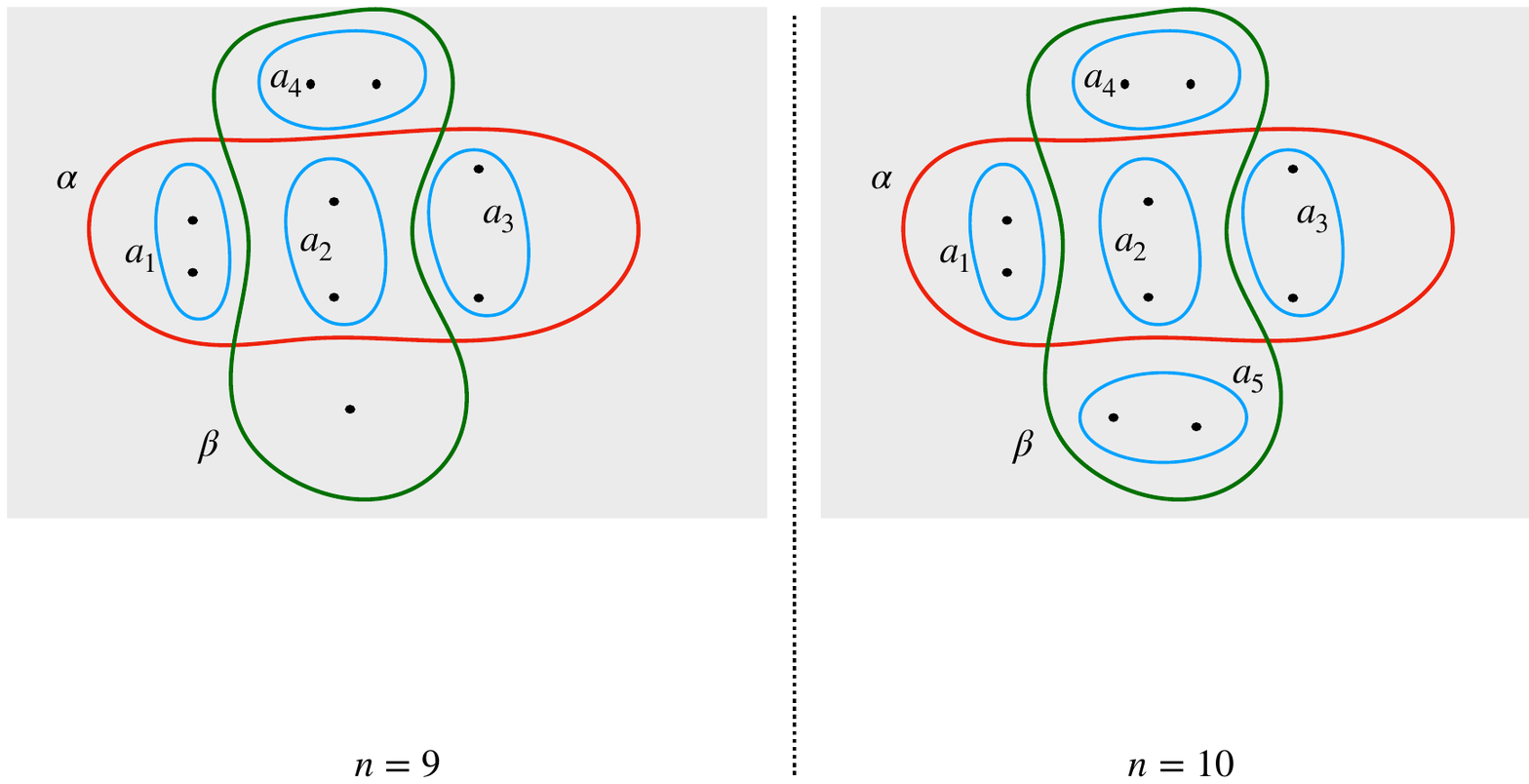}
	\caption{}
	\label{Fig.9,10}
\end{figure}

We verify that
		\begin{equation*}
			\mathfrak{Sp}_2(0, 9)=\{1, 2, 3, 4\}= \left \{1, 2, 3, \left \lfloor \frac{9}{2}\right \rfloor\right \}.
		\end{equation*}
		\item $n=10$: Since $\mathfrak{Sp}_2(0, 8)=\{1, 2, 3, 4\}$, it follows from (\ref{OpI-0}) that $\mathfrak{Sp}_2(0, 10)\supset \{1, 2, 3, 4\}$.
The  picture of the right side in Figure \ref{Fig.9,10} shows that $\mathfrak{Sp}_2(0, 10)\ni 5$.
We verify that
		\begin{equation*}
			\mathfrak{Sp}_2(0, 10)=\{1, 2, 3, 4, 5\}= \left \{1, 2, 3, 4, \left \lfloor \frac{10}{2}\right \rfloor\right \}.
		\end{equation*}
	\end{itemize}
	The proof of the lemma is completed.
\end{proof}

Next, we consider the case where $g=1$.
\begin{lemma}
\label{torus Lemma}
	If $n\geq 4$, then
	\begin{equation}
	\label{Lemma3}
		\mathfrak{Sp}_2(1, n)=\left \{1, 2, \dots , \left \lfloor \frac{3+n}{2}\right \rfloor\right \}.
	\end{equation}
\end{lemma}
\begin{proof}
	From the same inductive argument as in the proof of the above lemma, we see that it suffices to show the statement for $n=4, 5, 6, 7$.
	\begin{itemize}
		\item $n=4$: Since $\mathfrak{Sp}_2(1, 2)=\{1\}$, it follows from (\ref{OpI-0}) that $\mathfrak{Sp}_2(1, 4)\ni 1$. 
Figure \ref{Fig.(1,4)} shows that $\mathfrak{Sp}_2(1, 4)\ni 2, 3$.

\begin{figure}[htbp]
	\centering
	\includegraphics[width=12cm]{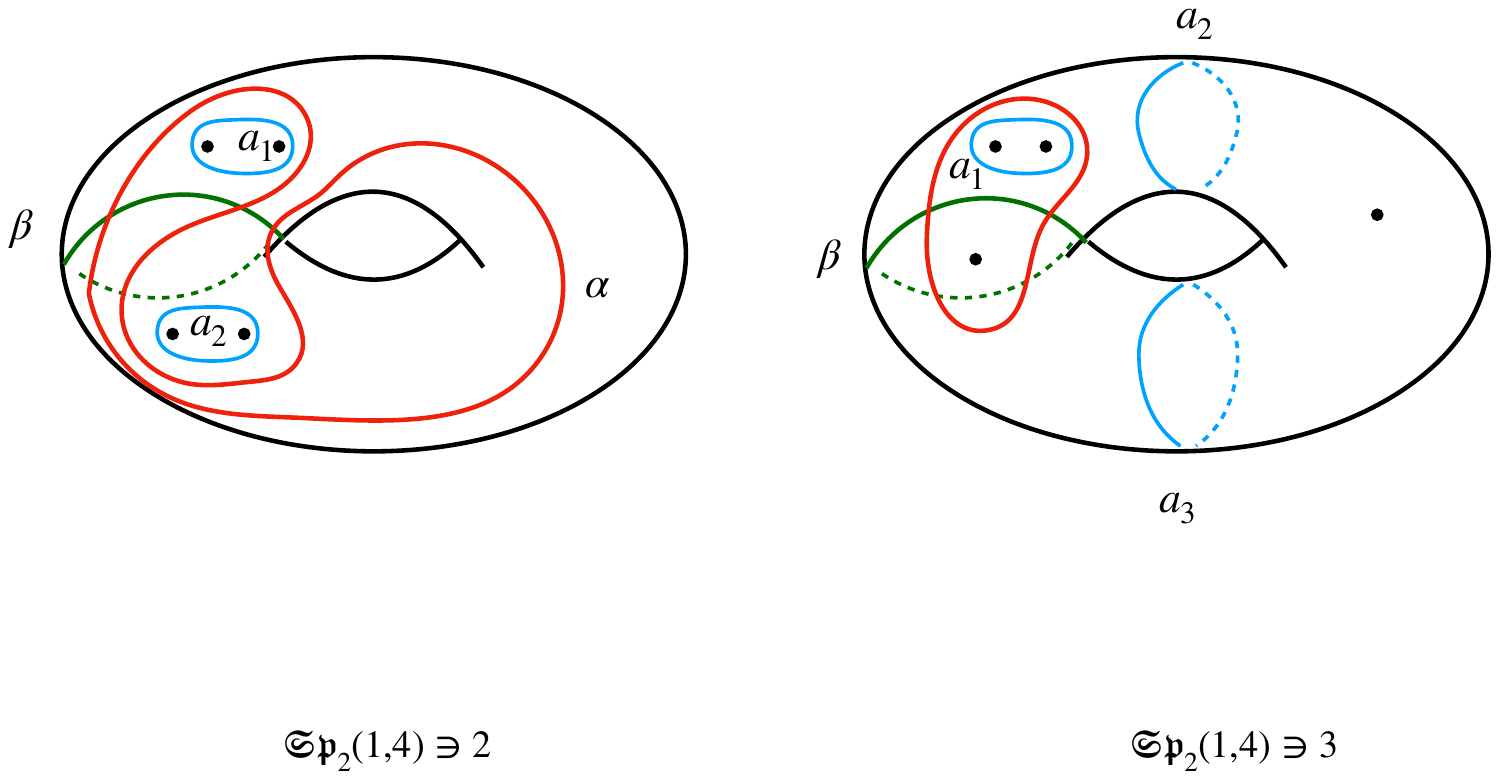}
	\caption{The $(1, 4)$ spectrum}
	\label{Fig.(1,4)}
\end{figure}

We verify that
		\begin{equation*}
			\mathfrak{Sp}_2(1, 4)=\{1, 2, 3\}= \left \{1, 2, \left \lfloor \frac{3+4}{2}\right \rfloor\right \}.
		\end{equation*}
		\item $n=5$: Since $\mathfrak{Sp}_2(1, 3)=\{1, 2\}$, it follows from (\ref{OpI-0}) that $\mathfrak{Sp}_2(1, 5)\supset \{1, 2\}$. 
Figure \ref{Fig.(1,5)} shows that $\mathfrak{Sp}_2(1, 5)\ni 3, 4$.

\begin{figure}[htbp]
	\centering
	\includegraphics[width=12cm]{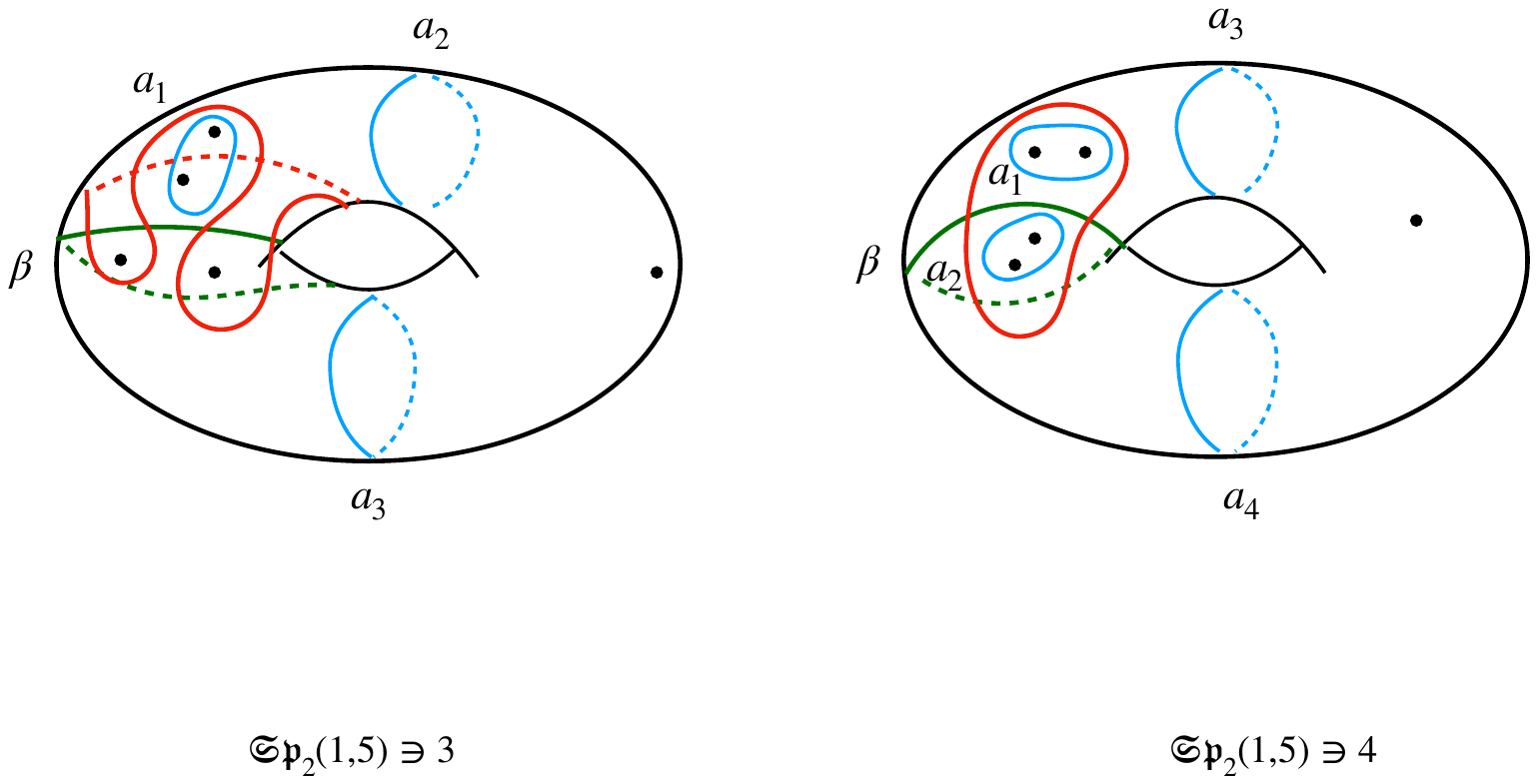}
	\caption{The $(1, 5)$ spectrum}
	\label{Fig.(1,5)}
\end{figure}

We verify that
		\begin{equation*}
			\mathfrak{Sp}_2(1, 5)=\{1, 2, 3, 4\}= \left \{1, 2, 3, \left \lfloor \frac{3+5}{2}\right \rfloor\right \}.
		\end{equation*}
		\item $n=6$: Since $\mathfrak{Sp}_2(1, 4)=\{1, 2, 3\}$, it follows from (\ref{OpI-0}) that $\mathfrak{Sp}_2(1, 6)\supset \{1, 2, 3\}$. 
The picture of the left side in Figure \ref{Fig.(1,6)} shows that $\mathfrak{Sp}_2(1, 6)\ni 4$.

\begin{figure}[htbp]
	\centering
	\includegraphics[width=12cm]{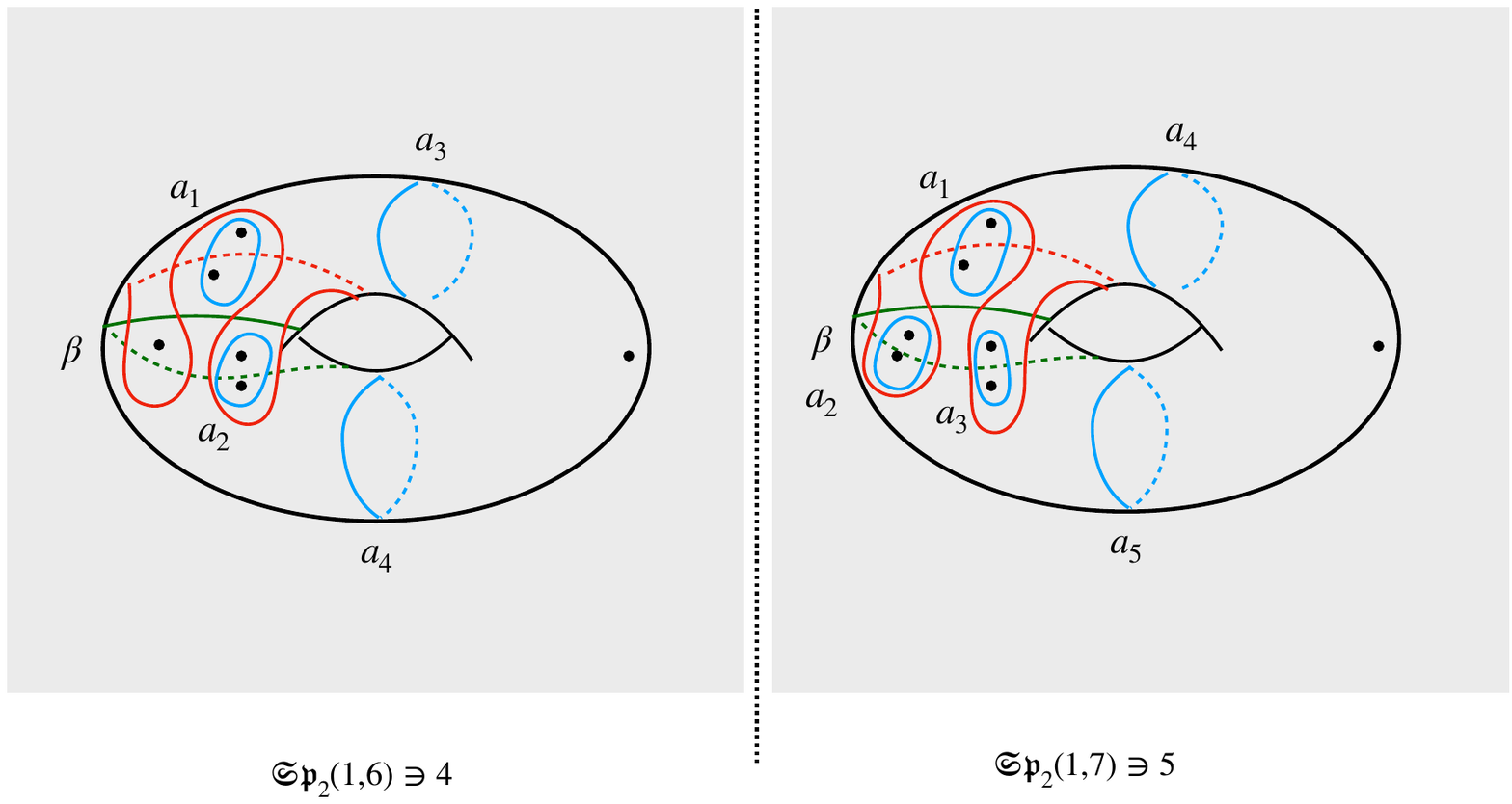}
	\caption{}
	\label{Fig.(1,6)}
\end{figure}

We verify that
		\begin{equation*}
			\mathfrak{Sp}_2(1, 6)=\{1, 2, 3, 4\}= \left \{1, 2, 3, \left \lfloor \frac{3+6}{2}\right \rfloor\right \}.
		\end{equation*}
		\item $n=7$: Since $\mathfrak{Sp}_2(1, 5)=\{1, 2, 3, 4\}$, it follows from (\ref{OpI-0}) that $\mathfrak{Sp}_2(1, 7)\supset \{1, 2, 3, 4\}$. 
The picture of the right side in Figure \ref{Fig.(1,6)} shows that $\mathfrak{Sp}_2(1, 7)\ni 5$.
We verify that
		\begin{equation*}
			\mathfrak{Sp}_2(1, 7)=\{1, 2, 3, 4, 5\}= \left \{1, 2, 3, 4, \left \lfloor \frac{3+7}{2}\right \rfloor\right \}.
		\end{equation*}
	\end{itemize}
	Thus, we have completed the proof of the lemma.
\end{proof}

We also consider the case where $g=2$.
\begin{lemma}
\label{genus2 Lemma}
	If $n\geq 1$, then
\begin{equation}
	\label{Lemma4}
		\mathfrak{Sp}_2(2, n)=\left \{1, 2, \dots , \left \lfloor \frac{6+n}{2}\right \rfloor\right \}.
	\end{equation}
\end{lemma}
\begin{proof}
	From the same inductive argument as in the previous lemmas, it suffices to show the statement for $n=1, 2, 3, 4$.
	\begin{itemize}
		\item $n=1$: We have already shown that $\mathfrak{Sp}_2(2, 1)\supset\{1, 2, 3\}$ in Figure \ref{Fig.(2,1)}.
Since $M_2(2, 1)=m(2, 1)=3$, we see that 
\begin{equation*}
	\mathfrak{Sp}_2(2, 1)=\{1, 2, 3\}= \left \{1, 2, \left \lfloor \frac{6+1}{2}\right \rfloor\right \}.
\end{equation*}
\item $n=2$: Since $\mathfrak{Sp}_2(2, 0)=\{1, 2\}$, it follows from (\ref{OpI-0}) that $\mathfrak{Sp}_2(2, 2)\supset\{1, 2\}$.
Figure \ref{Fig.(2,2)} shows that $\mathfrak{Sp}_2(2, 2)\ni 3, 4$.
We verify that
		\begin{equation*}
			\mathfrak{Sp}_2(2, 2)=\{1, 2, 3, 4\}= \left \{1, 2, 3, \left \lfloor \frac{6+2}{2}\right \rfloor\right \}.
		\end{equation*}
\begin{figure}[htbp]
	\centering
	\includegraphics[width=12cm]{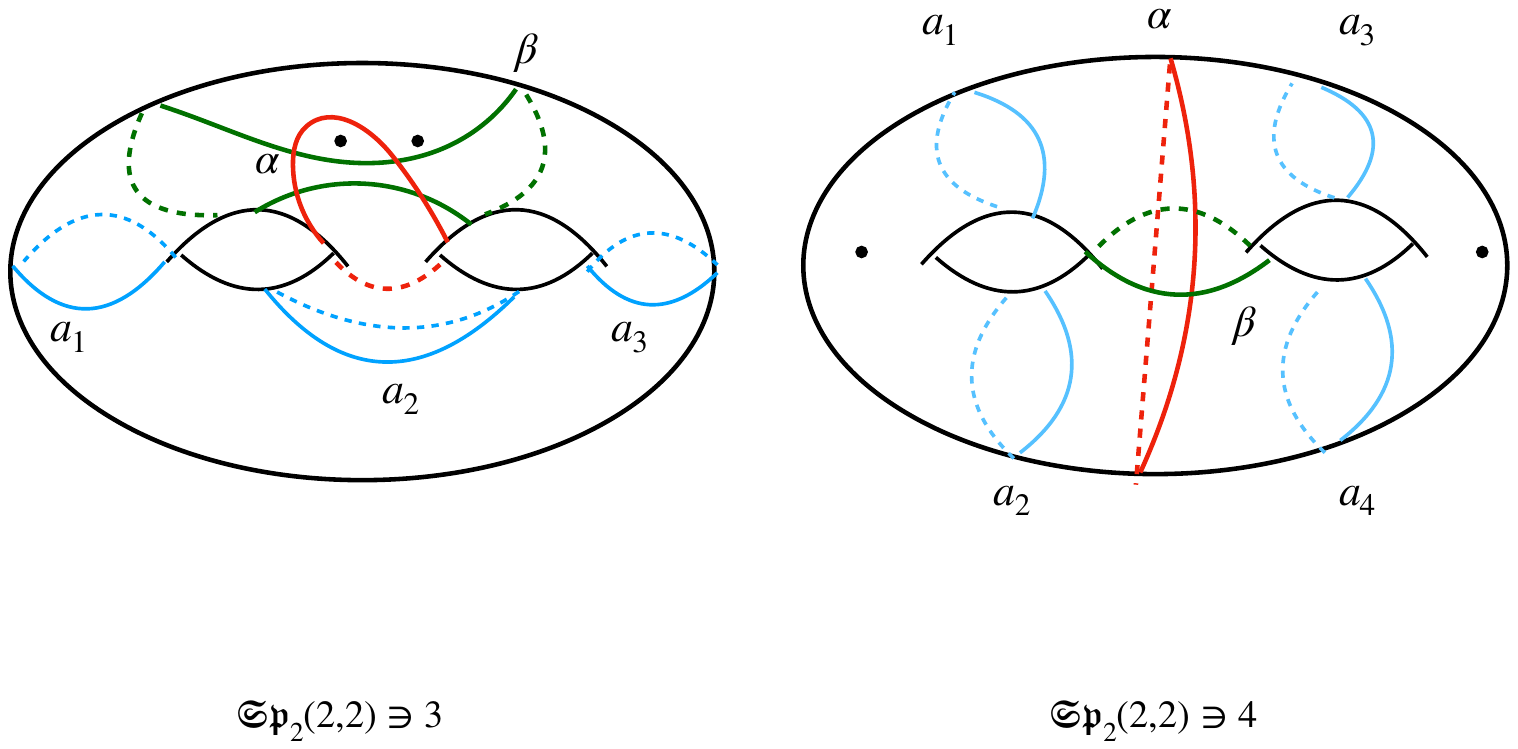}
	\caption{The $(2, 2)$ spectrum}
	\label{Fig.(2,2)}
\end{figure}

\item $n=3$: Since $\mathfrak{Sp}_2(2, 1)=\{1, 2, 3\}$, it follows from (\ref{OpI-0}) that $\mathfrak{Sp}_2(2, 3)\supset\{1, 2, 3\}$.
The picture of the left side in Figure \ref{Fig.(2,3)} shows that $\mathfrak{Sp}_2(2, 3)\ni 4$.
We verify that
		\begin{equation*}
			\mathfrak{Sp}_2(2, 3)=\{1, 2, 3, 4\}= \left \{1, 2, 3, \left \lfloor \frac{6+3}{2}\right \rfloor\right \}.
		\end{equation*}
		
\item $n=4$: Since $\mathfrak{Sp}_2(2, 2)=\{1, 2, 3, 4\}$, it follows from (\ref{OpI-0}) that $\mathfrak{Sp}_2(2, 3)\supset\{1, 2, 3, 4\}$.
The picture of the right side in Figure \ref{Fig.(2,3)} shows that $\mathfrak{Sp}_2(2, 4)\ni 5$.
We verify that
		\begin{equation*}
			\mathfrak{Sp}_2(2, 4)=\{1, 2, 3, 4, 5\}= \left \{1, 2, 3, \left \lfloor \frac{6+4}{2}\right \rfloor\right \}.
		\end{equation*}

\begin{figure}[htbp]
	\centering
	\includegraphics[width=12cm]{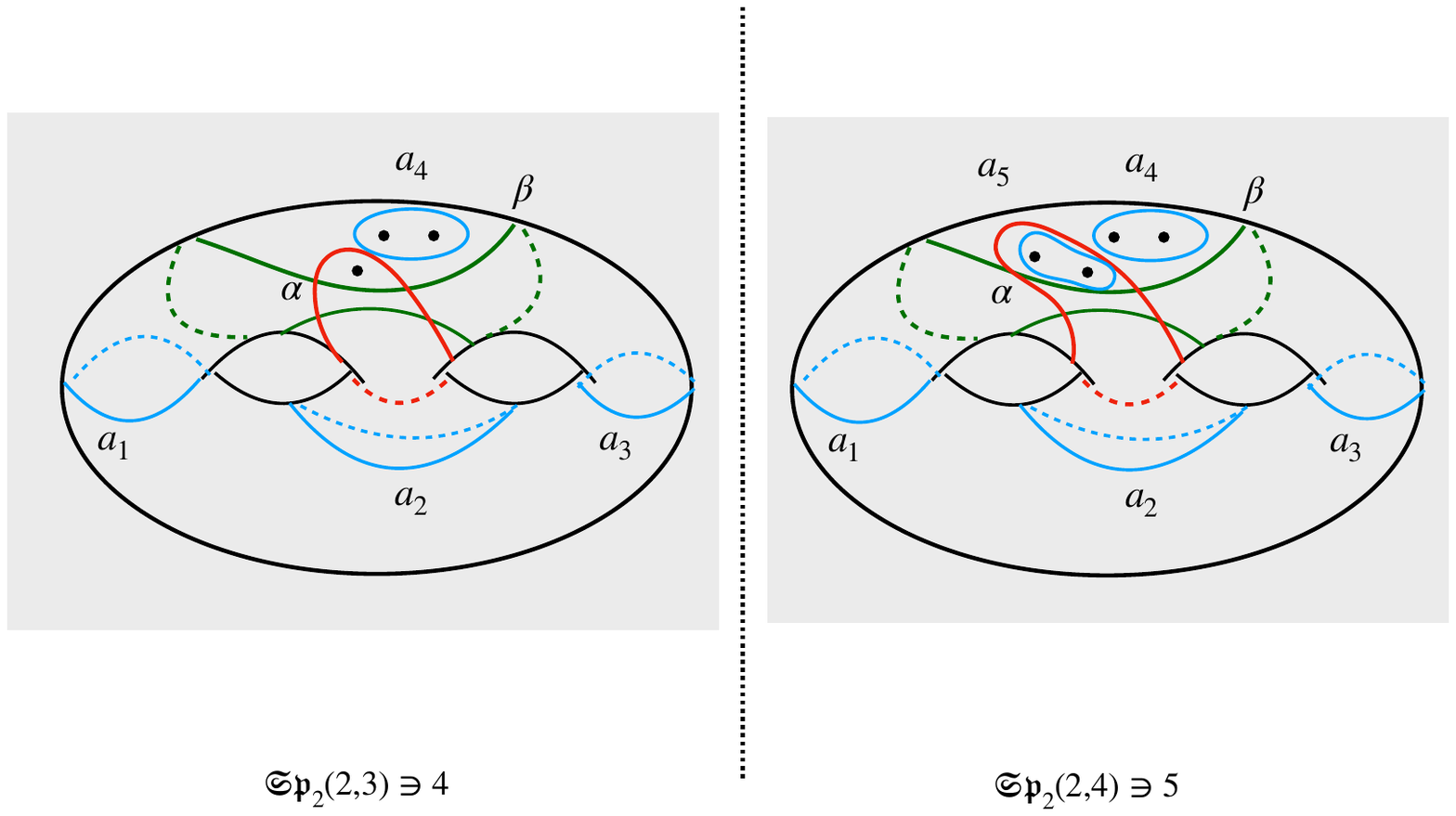}
	\caption{}
	\label{Fig.(2,3)}
\end{figure}
	\end{itemize}
Thus, the proof is completed.
\end{proof}

\begin{lemma}
\label{lemma5}
	If $g\geq 3$ and $n\geq g-3$, then
\begin{equation}
\label{eqn:half}
	\mathfrak{Sp}_2(g, n)=\left\{1, 2, \dots , \left\lfloor \frac{3g+n}{2}\right\rfloor\right\}.
\end{equation}
\end{lemma}
\begin{proof}
	First, we show (\ref{eqn:half}) for $(g, n)=(3, 0)$ and $(3, 1)$.
	
	We have already known that $\mathfrak{Sp}_2(3, 0), \mathfrak{Sp}_2(3, 1)\supset \{1, 2, 3\}$ (Lemma \ref{FirstEstimate}).
	From Lemma \ref{Lemma4}, we also know that $\mathfrak{Sp}_2(2, 1)\ni 4$.
	Therefore, $\mathfrak{Sp}_2(3, 2)\ni 4$ because of (\ref{OpII}) of Operation II.
	
	The picture of the left side in Figure \ref{Fig.(3,0)} shows $\mathfrak{Sp}_2(3,0)\ni 4$ and the right side one shows $\mathfrak{Sp}_2(3, 1)\ni 5$.
	Since $\lfloor\frac{3\cdot3}{2}\rfloor=4$ and $\lfloor\frac{3\cdot 3+1}{2}\rfloor =5$, we obtain the desired result for $(g, n)=(3, 0)$ and $(3, 1)$.
	\begin{figure}[htbp]
	\centering
	\includegraphics[width=12cm]{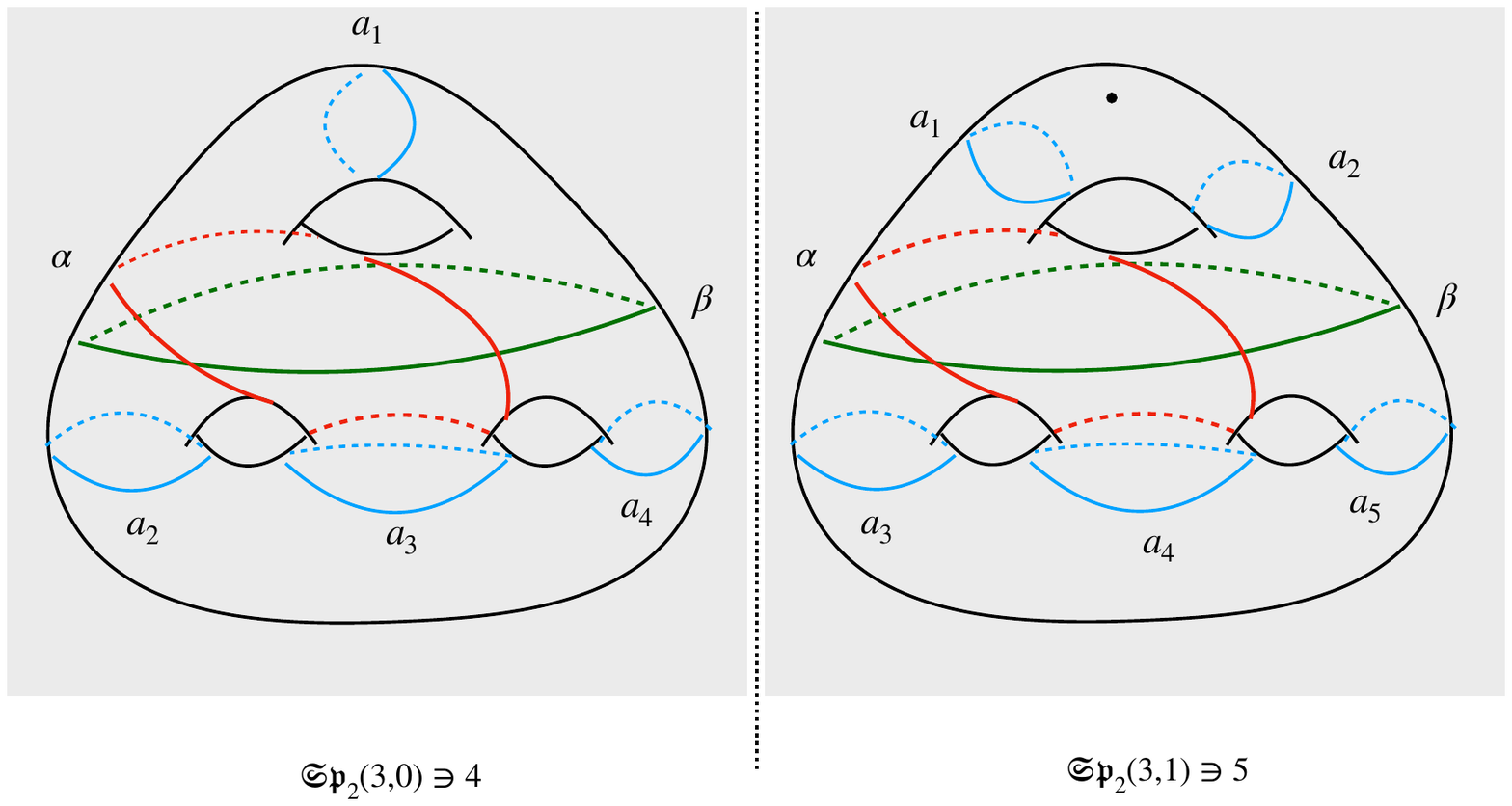}
	\caption{}
	\label{Fig.(3,0)}
\end{figure}

Next, we consider the case where $(g, n)=(3, n)$ for $n\geq 2$.

From Lemma \ref{FirstEstimate} and Lemma \ref{Lemma4}, we have
\begin{equation*}
	\mathfrak{Sp}_2(3, n)\supset\{1, 2, 3\}, \ \mathfrak{Sp}_2(2, n)=\left\{1, 2, \dots , \left\lfloor\frac{6+n}{2}\right\rfloor\right\}
\end{equation*}
if $n\geq 1$.
It follows from (\ref{OpIII}) of Operation III that
\begin{eqnarray*}
		\mathfrak{Sp}_2(3, n)&=&\mathfrak{Sp}_2(2+1, (n-1)+1)\supset\left\{1, \dots , \left\lfloor\frac{5+n}{2}\right\rfloor\right\}+2 \\
		&=&\left\{3, \dots , \left\lfloor\frac{9+n}{2}\right\rfloor\right\}.
\end{eqnarray*}
Thus, we obtain
\begin{equation*}
	\mathfrak{Sp}_2(3, n)\supset\{1, 2, 3\}\cup\left\{3, \dots , \left\lfloor\frac{9+n}{2}\right\rfloor\right\}=\left\{1, \dots , \left\lfloor\frac{9+n}{2}\right\rfloor\right\}.
\end{equation*}
Since $M_2(3, n)=\lfloor\frac{9+n}{2}\rfloor$, overall we verify that
\begin{equation*}
	\mathfrak{Sp}_2(3, n)=\left\{1, \dots , \left\lfloor\frac{9+n}{2}\right\rfloor\right\}
\end{equation*}
for $n\geq 0$.

The same argument works to show (\ref{eqn:half}) for $(g, n)=(4, n)$ ($n\geq 1$).
Indeed, noting that $\mathfrak{Sp}_2(4, n)\supset\mathfrak{Sp}_2(3, n-1)+2$, we see
\begin{eqnarray*}
		\mathfrak{Sp}_2(4, n)&=&\mathfrak{Sp}_2(3+1, (n-1)+1)\supset\left\{1, \dots , \left\lfloor\frac{8+n}{2}\right\rfloor\right\}+2 \\
		&=&\left\{3, \dots , \left\lfloor\frac{12+n}{2}\right\rfloor\right\}.
\end{eqnarray*}
Hence, we conclude that
\begin{equation*}
	\mathfrak{Sp}_2(4, n)=\left\{1, \dots , \left\lfloor\frac{12+n}{2}\right\rfloor\right\}
\end{equation*}
by the same reason as above.

Repeating these arguments, we may show that (\ref{eqn:half}) is valid if $g\ge 3$ and $n\geq g-3$.
\end{proof}

Now, we consider the rest of half to complete the proof of Theorem \ref{ThmI}.

\begin{lemma}
\label{Last Lemma}
	If $g\geq 4$ and $0\leq n< g-3$, then
	\begin{equation}
	\label{eqn:half2}
		\mathfrak{Sp}_2(g, n)=\left\{1, \dots , \left\lfloor\frac{3g+n}{2}\right\rfloor\right\}.
	\end{equation}
\end{lemma}
\begin{proof}
First, we show the statement for $(g, 0)$ $(g\geq 4)$.
 To show it, we note an inductive argument similar to the previous arguments.

Suppose that (\ref{eqn:half2}) is true for some $(g, 0)$ with $g\geq 3$.
Then, we have
\begin{equation*}
	\mathfrak{Sp}_2(g+4, 0)\supset \mathfrak{Sp}_2(g, 0)+6,
\end{equation*}
from (\ref{OpIV}) of Operation IV.
Since $\lfloor\frac{3g}{2}\rfloor+6=\lfloor\frac{3(g+4)}{2}\rfloor$,
we obtain
\begin{equation*}
	\mathfrak{Sp}_2(g+4, 0)\supset\left\{7, \dots , \left\lfloor\frac{3(g+4)}{2}\right\rfloor\right\}.
\end{equation*}
From Lemma \ref{FirstEstimate}, we know that
\begin{equation*}
	\mathfrak{Sp}_2(g+4, 0)\supset\{1, 2, \dots , g+4\}\supset\{1, 2, \dots, 7\}.
\end{equation*}
Thus, we have
\begin{eqnarray*}
	\mathfrak{Sp}_2(g+4, 0)&\supset &\{1, 2, \dots, 7\}\cup\left\{7, \dots , \left\lfloor\frac{3(g+4)}{2}\right\rfloor\right\} \\
	&=& \left\{1, \dots , \left\lfloor\frac{3(g+4)}{2}\right\rfloor\right\}.
\end{eqnarray*}
Since $M_2(g+4, 0)=\lfloor\frac{3(g+4)}{2}\rfloor$, we verify that
\begin{equation*}
	\mathfrak{Sp}_2(g+4, 0)=\left\{1, \dots , \left\lfloor\frac{3(g+4)}{2}\right\rfloor\right\}.
\end{equation*}
We have already shown that the statement is true for $(3, 0)$ in the previous lemma.
Therefore, it suffices to show (\ref{eqn:half2}) for the case where $g=4, 5, 6$.

\begin{itemize}
	\item $g=4$: As we have already known that  $\mathfrak{Sp}_2(4, 0), \ \mathfrak{Sp}_2(3, 0) \supset \{1, 2, 3, 4\}$, we have
\begin{equation*}
	\mathfrak{Sp}_2(4, 0)\supset\{1, 2, 3, 4, 5\}
\end{equation*}
from (\ref{OpII}) of Operation II.
The picture of the left side in Figure \ref{Fig.(4,0)} shows that $\mathfrak{Sp}_2(4, 0)\ni 6=\lfloor\frac{3\cdot 4}{2}\rfloor$.

	\begin{figure}
	\centering
	\includegraphics[width=12cm]{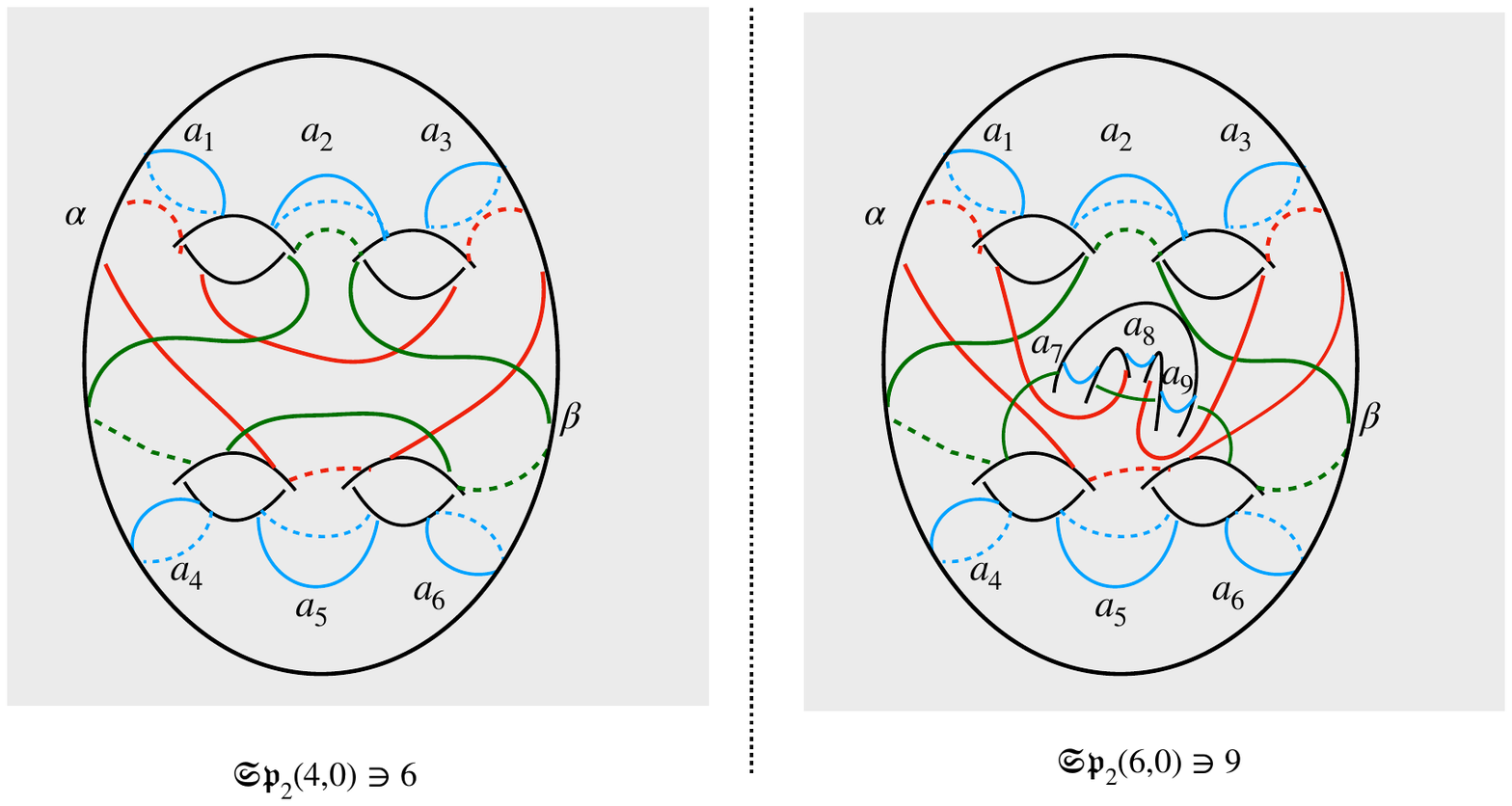}
	\caption{}
	\label{Fig.(4,0)}
\end{figure}

Thus, we have
\begin{equation*}
	\mathfrak{Sp}_2(4, 0)=\{1, 2, 3, 4, 5, 6\}.
\end{equation*}
\item $g=5$: Applying Operation II to $\mathfrak{Sp}_2(4, 0)$, we obtain $\mathfrak{Sp}_2(5, 0)\ni 7=\lfloor\frac{3\cdot 5}{2}\rfloor$.
Hence, we verify that
\begin{equation*}
	\mathfrak{Sp}_2(5, 0)=\{1, 2, 3, 4, 5, 6, 7\}.
\end{equation*}
\item $g=6$: We have $\mathfrak{Sp}_2(6, 0)\supset\{1, 2, 3, 4, 5, 6, 7, 8\}$ as before.
A bit complicated picture of the right side in Figure \ref{Fig.(4,0)} shows that $\mathfrak{Sp}_2(0, 6)\ni 9=\lfloor\frac{3\cdot 6}{2}\rfloor$.
Hence, we verify that
\begin{equation*}
	\mathfrak{Sp}_2(6, 0)=\{1, 2, 3, 4, 5, 6, 7, 8, 9\}.
\end{equation*}
\end{itemize}
We prove that (\ref{eqn:half2}) is valid for $(g, 0)$ $(g\geq 4)$.

To complete the proof of the lemma, we use Operation III.
We may show the statement for $(g, 1)$ $(g\geq 4)$.
Indeed, as $(g, 1)=((g-1)+1, 0+1)$, it follows from (\ref{OpIII}) of Operation III that
\begin{eqnarray*}
	\mathfrak{Sp}_2(g, 1)&\supset&\mathfrak{Sp}_2(g-1, 0)+2\supset\left\{3, 4, \dots , \left\lfloor\frac{3(g-1)}{2}\right\rfloor +2\right\} \\
	&=&\left\{3, 4, \dots , \left\lfloor\frac{3g+1}{2}\right\rfloor\right\}.
\end{eqnarray*}
Thus, we obtain
\begin{equation*}
	\mathfrak{Sp}_2(g, 1)=\left\{1, 2, \dots , \left\lfloor\frac{3g+1}{2}\right\rfloor\right\}
\end{equation*}
as before.
Repeating this argument, we verify that
\begin{equation*}
	\mathfrak{Sp}_2(g, n)=\left\{1, 2, \dots , \left\lfloor\frac{3g+n}{2}\right\rfloor\right\}
\end{equation*}
for $n\in\mathbb N$ with $n\leq g-3$ and the proof of the lemma is completed.
\end{proof}

Combining Lemmas \ref{planar Lemma}, \ref{torus Lemma}, \ref{genus2 Lemma},  \ref{lemma5} and \ref{Last Lemma}, we complete the proof of Theorem \ref{ThmI} for $\mathfrak{Sp}_2(g, n)$.
Since $\mathfrak{Sp}_2(g, n)\subset\mathfrak{Sp}_2^T(g, n)$ and $M_2(g, n)=M_2^T(g, n)$ (see Remark \ref{Remark}), we obtain $\mathfrak{Sp}_2(g, n)=\mathfrak{Sp}_2^T(g, n)$. It shows that Theorem \ref{ThmI} is true for $\mathfrak{Sp}_2^T(g, n)$. \qed

\end{document}